\documentclass[11pt,a4paper]{article}
\usepackage[T1]{fontenc}
\usepackage{lmodern,amsmath,amssymb,amsthm,array}
\usepackage[margin=26mm]{geometry}
\usepackage[colorlinks=true,linkcolor=blue,citecolor=blue,urlcolor=blue,hypertexnames=false]{hyperref}
\DeclareMathOperator{\diam}{diam}
\DeclareMathOperator{\dist}{dist}
\DeclareMathOperator{\rad}{rad}
\DeclareMathOperator{\ecc}{ecc}
\makeatletter
\newcommand{\rightharpoonupfill@}{\arrowfill@\relbar\relbar\rightharpoonup}
\DeclareRobustCommand{\rbjt}[1]{\mathpalette{\overarrow@\rightharpoonupfill@}{#1}}
\makeatother
\newtheorem{lemma}{Lemma}[section]
\newtheorem{theorem}[lemma]{Theorem}
\title{An asymptotically tight bound on oriented diameter}
\author{Jiangdong Ai\thanks{School of Mathematical Sciences and LPMC, Nankai University, Tianjin 300071, China. \texttt{jd@nankai.edu.cn}}
\and Yaojun Chen\thanks{School of Mathematics, Nanjing University, Nanjing 210093, China. \texttt{yaojunc@nju.edu.cn}}
\and Yaokun Feng\thanks{Center for Combinatorics and LPMC, Nankai University, Tianjin 300071, China and Department of Applied Mathematics, The Hong Kong Polytechnic University, Hong Kong, China. \texttt{1031679547@qq.com}}
\and Hui Lei\thanks{School of Statistics and Data Science, LPMC and KLMDASR, Nankai University, Tianjin 300071, China. \texttt{hlei@nankai.edu.cn}}
\and Jifu Lin\thanks{Department of Mathematics, East China Normal University, Shanghai 200241, China. \texttt{jifulin01@163.com}}
\and Zijian Ren\thanks{School of Mathematical Sciences and LPMC, Nankai University, Tianjin 300071, China. \texttt{2120250113@mail.nankai.edu.cn}}
\and Xiaolin Wang\thanks{School of Mathematics and Statistics, Fuzhou University, Fuzhou 350108, China. \texttt{xiaolinw92@163.com}}
\and Ruilin Zheng\thanks{School of Mathematics,  Nanjing University, Nanjing 210093, China. \texttt{rlzhengmath@163.com}}}
\date{}
\begin{document}
\maketitle

\begin{abstract}
Let $f(d)$ be the smallest integer such that every finite connected bridgeless simple graph of diameter $d$ admits a strong orientation of diameter at most $f(d)$. Chv\'atal and Thomassen (1978) proved $\lceil d^2/2\rceil+d\le f(d)\le2d^2+2d$. In this paper, we prove that $\lceil d^2/2\rceil+d\le f(d)\le\lceil d^2/2\rceil+d+18$ for every $d\ge2$, which shows that $f(d)=\lceil d^2/2\rceil+d+O(1)$ and determines both the quadratic and linear terms up to a bounded additive error.

The key method of our proof is to construct a central subgraph $H$ that admits a strong orientation of diameter $O(d)$ and is within distance $\lfloor d/2\rfloor$ of every vertex outside it. We obtain the sharp linear coefficient by jointly estimating outside paths and their connecting paths in $H$ at the actual attachment vertices, rebuilding $H$ when necessary.
\end{abstract}

\textbf{Keywords.} Strong orientation, oriented diameter, diameter.

\section{Introduction}

All graphs in this paper are finite and undirected unless an orientation has been specified. We allow parallel edges but no loops, and treat parallel edges as distinct edges. A graph is simple if it has no parallel edges. The main result concerns simple graphs, although auxiliary graphs may have parallel edges. For a connected graph $G$, the distance $\dist_G(x,y)$ is the length of a shortest $x$--$y$ path, and its diameter $\diam(G)$ is the maximum of these distances. An edge is a bridge if its deletion disconnects the graph, and a graph is bridgeless if it has no bridge. An orientation $\rbjt{G}$ assigns one direction to each edge independently; in particular, parallel edges in an auxiliary graph may receive opposite directions. The orientation is strong if every ordered pair of vertices can be joined by a directed path. In 1939, Robbins~\cite{Robbins1939} proved that a connected graph has a strong orientation if and only if it is bridgeless. For a strong orientation $\rbjt{G}$, the directed distance $\dist_{\rbjt{G}}(x,y)$ is the length of a shortest directed $x$--$y$ path, and the directed diameter $\diam(\rbjt{G})$ is the maximum directed distance. The oriented diameter of a connected bridgeless graph is $\rbjt{\diam}(G)=\min_{\rbjt{G}}\diam(\rbjt{G})$, where the minimum is over all strong orientations of $G$.

By Robbins' theorem, a natural question is how small the diameter of a strong orientation can be, and whether it can be bounded in terms of the diameter of the underlying graph. Clearly, $\rbjt{\diam}(G)\ge\diam(G)$. Chv\'atal and Thomassen~\cite{ChvatalThomassen1978} initiated the study of this question. For an integer $d\ge0$, let $f(d)$ be the smallest integer such that every finite connected bridgeless simple graph of diameter $d$ admits a strong orientation of diameter at most $f(d)$. Equivalently, $f(d)$ is the largest oriented diameter among these graphs. Chv\'atal and Thomassen~\cite{ChvatalThomassen1978} proved the bounds $\lceil d^2/2\rceil+d\le f(d)\le2d^2+2d$ for $d\ge2$. The lower bound comes from a construction, and the upper bound follows from their radius theorem. The lower-bound construction for odd $d$ was subsequently corrected by Kwok, Liu and West~\cite{KwokLiuWest2010}. Determining the coefficient of $d^2$ has therefore been a main question in the study of oriented diameter.

In 2021, Babu, Benson, Rajendraprasad and Vaka~\cite{BabuEtAl2021} improved the upper bound for $d\ge8$ to $f(d)\le1.373d^2+6.971d-1$. More recently, Ai, Lei and Ren~\cite{AiLeiRen2026} proved $f(d)\le d^2+\lceil11d/2\rceil-5$ for $d\ge3$. These results reduced the gap in the general bounds to $$\lceil d^2/2\rceil+d\le f(d)\le d^2+\lceil11d/2\rceil-5.$$
Alongside these general bounds, the exact values for small diameters have received considerable attention. Chv\'atal and Thomassen~\cite{ChvatalThomassen1978} proved $f(2)=6$. For diameter $3$, Kwok, Liu and West~\cite{KwokLiuWest2010} established $9\le f(3)\le11$, and Wang and Chen~\cite{WangChen2022} subsequently determined $f(3)=9$. The next case, diameter $4$, has seen several recent advances~\cite{BabuEtAl2021,FengEtAl2026,LinYou2026,LinWangYou2026}. Feng, Lei, Lian and Ren~\cite{FengEtAl2026} proved $f(4)\le16$. Together with the general lower bound, this gives $12\le f(4)\le16$, leaving the exact value undetermined. These results show that determining $f(d)$ for an individual small value of $d$ requires sharper estimates than the general quadratic bounds provide.

For simple graphs, oriented diameter has also been studied in terms of other graph parameters. Bounds involving the minimum degree were obtained by Bau and Dankelmann~\cite{BauDankelmann2015} and Surmacs~\cite{Surmacs2017}, with recent progress by Cochran and Wang~\cite{CochranWang2026}. Bounds involving the maximum degree, and the maximum degree together with girth, appear in~\cite{DankelmannGuoSurmacs2018,ChenChang2023}. The domination number $\gamma(G)$ is the minimum size of a vertex set containing a neighbour of every vertex outside it. Fomin, Matamala, Prisner and Rapaport~\cite{FominEtAl2004} proved $\rbjt{\diam}(G)\le9\gamma(G)-5$. Kurz and L\"atsch~\cite{KurzLatsch2012} improved this to $4\gamma(G)$, and Wang and Chen~\cite{WangChen2025} established the sharp bound $\lceil(7\gamma(G)+1)/2\rceil$. Connected domination and distance domination were considered by Dankelmann, Morgan and Rivett-Carnac~\cite{DankelmannEtAl2024}.

There is also an extensive literature on particular graph classes, including complete $k$-partite graphs for $k\ge2$~\cite{G,LS,P}, AT-free graphs~\cite{FominEtAl2004}, chordal graphs~\cite{FominMatamalaRapaport2004}, maximal outerplanar graphs~\cite{WangEtAl2021}, near planar triangulations~\cite{GeLiuWang2025}, and planar triangulations~\cite{Liu2026,Wang}. In this paper, we prove the following theorem.

\begin{theorem}\label{thm:main}
Every finite connected bridgeless simple graph $G$ of diameter $d$ satisfies
\[
\rbjt{\diam}(G)\le\left\lceil\frac{d^2}{2}\right\rceil+d+18.
\]
\end{theorem}

Together with the known lower bound, Theorem~\ref{thm:main} yields $\lceil d^2/2\rceil+d\le f(d)\le\lceil d^2/2\rceil+d+18.$ Consequently, $f(d)=\lceil d^2/2\rceil+d+O(1)$. 

The framework of our proof is inspired by the contraction method used in~\cite{BabuEtAl2021,WangChen2025}. The construction of the central subgraph is also motivated by the proof for graphs of diameter $3$ in~\cite{WangChen2022}, where the vertex set is partitioned according to the distances from the endpoints of a selected edge. Here we choose a centre $p$, meaning a vertex whose maximum distance to the other vertices is as small as possible, and a farthest vertex $q$ of $p$. We partition $V(G)$ into two parts according to their relative distances from $p$ and $q$, and fix a shortest $p$--$q$ path. Using parts of two fixed shortest-path trees, we join each crossing edge other than one specified edge to a segment of this path. The resulting directed cycles share the specified edge. Their union is a subgraph $H$ with a strong orientation, and it contains every edge between the parts. The two essential properties of $H$ are that every vertex outside $H$ is within distance $\lfloor d/2\rfloor$ of $H$, and $H$ has a strong orientation of diameter $O(d)$. We extend the orientation of $H$ using rooted auxiliary graphs for the components outside it. Their rooted path bounds supply the quadratic term. To obtain the linear coefficient $1$, we retain the actual attachment of every selected outside path and estimate its length together with the path joining the attachments in $H$. If the initial estimate is insufficient, a long selected path yields another farthest vertex of the same centre. Lower bounds for paths whose internal vertices lie in a suitable connected vertex set then allow us to rebuild $H$ with better bounds on its connecting paths. Among all choices obtainable by finitely many rebuildings, we choose one with minimum saving from the general rooted path bounds. The next rebuilding cannot decrease this saving, which supplies the comparison needed for the linear estimate. The specified-root form of the radius theorem of Chv\'atal and Thomassen~\cite{ChvatalThomassen1978}, stated explicitly by Dankelmann, Morgan and Rivett-Carnac~\cite{DankelmannEtAl2024}, handles the small-diameter and minimum-radius cases.

We finish this section with the notation used below. A path has no repeated vertex, whereas a walk may repeat vertices. Its length is its number of edges or arcs, and paths of length zero are allowed. For a path $Q$, the notation $|Q|$ means its length. The eccentricity of a vertex $u$ is $\ecc_G(u)=\max_v\dist_G(u,v)$, and the radius is $\rad(G)=\min_u\ecc_G(u)$. A centre attains the radius, and a geodesic is a shortest path. An attachment of a vertex set is a vertex outside it adjacent to a vertex in it. Distances without a subscript are in the original graph. We write $(x)_+=\max\{x,0\}$. After estimating the lengths of its segments, a directed walk may be shortened by deleting closed subwalks. The one-vertex graph has oriented diameter zero.

In the rest of this paper, Section~\ref{sec:central-core} constructs the initial central subgraph and its strong orientation. Section~\ref{sec:rooted-orientations} proves the rooted path bounds while retaining the attachment information. Section~\ref{sec:joining-paths} joins these paths through the central subgraph and identifies when its initial estimate suffices. Section~\ref{sec:rebuilding} obtains a new farthest vertex, rebuilds the central subgraph, and proves Theorem~\ref{thm:main}.

\section{A central  subgraph}\label{sec:central-core}

The main construction in this section is Lemma~\ref{lem:central-core}. We divide the vertices into two parts and construct a strongly oriented subgraph $H$ containing every edge between them. To construct $H$, we join each crossing edge other than one specified edge to a suitable segment of a fixed shortest path. The resulting directed cycles share that specified edge. Their union is therefore strong, and each component outside $H$ lies entirely in one part. Lemma~\ref{lem:shallow-attachments} improves the distances at attachments nearest the boundary, and Lemma~\ref{lem:axis-choice} controls the one exceptional pair on the fixed shortest path. Section~\ref{sec:joining-paths} combines these estimates with the outside paths.

Write $d=2r+\varepsilon$, where $r=\lfloor d/2\rfloor$ and $\varepsilon\in\{0,1\}$, and write $\rho=\rad(G)=r+t+\varepsilon$. Since $d\le2\rho$ and $\rho\le d$, we have $0\le t\le r$. Thus $t$ measures how far the radius exceeds its least possible value. Choose a centre $p$, a vertex $q$ with $\dist(p,q)=\rho$, and a shortest $p$--$q$ path $Q=Q_0\cdots Q_\rho$, where $Q_0=p$ and $Q_\rho=q$. The functions $a(v)=\dist(p,v)$ and $b(v)=\dist(q,v)$ satisfy $a(v)\le\rho$, $b(v)\le d$, and $a(v)+b(v)\ge\rho$. When $d$ is even, assume $t\ge1$; the case $t=0$ follows directly from the radius theorem in Theorem~\ref{thm:radius}.

Define $c=t-1+\varepsilon$ and $\theta=t-r-2+2\varepsilon$. The two parts are $X=\{v:a(v)-b(v)\le\theta\}$ and $Y=V(G)\setminus X$. This threshold is chosen to bound the distances from outside vertices to $H$ by $r-1+\varepsilon$ on $X$ and by $r$ on $Y$. The $X$ bound is one smaller when $d$ is even.
An edge with one end in each part is a crossing edge. Since $a(Q_i)=i$ and $b(Q_i)=\rho-i$, the vertex $Q_i$ belongs to $X$ precisely when $2i\le\rho+\theta$. The equality $\lfloor(\rho+\theta)/2\rfloor=c$ therefore shows that $Q_0,\ldots,Q_c$ lie in $X$ and $Q_{c+1},\ldots,Q_\rho$ lie in $Y$. Here $0\le c<\rho$, so $Q_cQ_{c+1}$ is the unique crossing edge on $Q$. We call it critical; every other crossing edge is noncritical.

We will also use two consequences of this partition. If $x\in X$, then $2b(x)=(a(x)+b(x))-(a(x)-b(x))\ge\rho-\theta=2r+2-\varepsilon$, so the integer $b(x)$ is at least $r+1$. If $y\in Y$, then $a(y)-b(y)\ge\theta+1$, and hence $2a(y)\ge\rho+\theta+1=2t-1+3\varepsilon$. Thus $a(y)\ge t+\varepsilon=c+1$.

\begin{lemma}\label{lem:central-core}
There is a connected bridgeless subgraph $H$ containing every crossing edge, with $Q\cap H$ a subpath of length $1\le h\le\rho$. It has a strong orientation whose only edge from $Y$ to $X$ is $Q_{c+1}\to Q_c$. Define $\varphi(v)=a(v)$ for $v\in X$ and $\varphi(v)=d+\rho-b(v)$ for $v\in Y$. This orientation satisfies
\begin{equation}\label{eq:core-distance}
\dist_{\rbjt{H}}(u,v)\le\varphi(v)-\varphi(u)+d+2h\qquad(u,v\in V(H)).
\end{equation}
Every component of $G-V(H)$ and all its attachments lie in one part.
\end{lemma}

\begin{proof}
We first specify a cycle for each noncritical crossing edge. We then check that their directions agree and that their union has the required distance bound.

Fix a shortest-path tree rooted at $p$ and a shortest-path tree rooted at $q$, both containing $Q$. Such trees exist because $Q$ is a shortest path: at each vertex outside $Q$, choose an edge to the preceding distance layer, and on $Q$ choose its preceding edge towards the relevant root. Following the chosen edges decreases the distance at every step and therefore reaches the root. Each nonroot vertex chooses one distinct edge, so there are $|V(G)|-1$ chosen edges. They consequently form a spanning tree, and every tree path to the root is shortest. The chosen trees remain fixed for all crossing edges.

Consider a crossing edge $xy$ with $x\in X$ and $y\in Y$. Along the $p$--$x$ tree path towards $x$, the value of $a$ increases by one at each edge, while the value of $b$ changes by at most one. Therefore $a-b$ never decreases along this path. Since its final vertex $x$ has $a(x)-b(x)\le\theta$, every vertex of the path lies in $X$. Along the $y$--$q$ tree path towards $q$, the value of $b$ decreases by one, while $a$ changes by at most one. Consequently, $a-b$ again never decreases. Since $y\in Y$, every vertex of this path lies in $Y$.

Suppose now that $xy$ is noncritical. Let $Q_\lambda$ be the last vertex of $Q$ on the $p$--$x$ tree path, and let $Q_\nu$ be the first vertex of $Q$ on the $y$--$q$ tree path. Because the first tree contains $Q$, its path from $p$ to $Q_\lambda$ is exactly $Q_0\cdots Q_\lambda$. Similarly, the second tree path from $Q_\nu$ to $q$ is $Q_\nu\cdots Q_\rho$. The two chosen paths therefore meet $Q$ in an initial segment and a final segment, respectively. Their parts outside $Q$ have no vertex of $Q$ in their interiors. Moreover, $Q_\lambda\in X$ and $Q_\nu\in Y$, so $\lambda\le c<\nu$.

Follow the first tree from $Q_\lambda$ to $x$, traverse $xy$, and follow the second tree from $y$ to $Q_\nu$. This is a path: its first tree segment lies in $X$, its second tree segment lies in $Y$, and each segment has no repeated vertex. Its only vertices on $Q$ are its two ends. Joining it to the segment $Q_\lambda\cdots Q_\nu$ of $Q$ therefore forms a cycle. Orient the tree segments and $xy$ from $Q_\lambda$ towards $Q_\nu$, and orient the segment of $Q$ from $Q_\nu$ back towards $Q_\lambda$. This is a directed cycle containing $Q_{c+1}\to Q_c$.

Perform this construction for every noncritical crossing edge. These directions agree on shared edges. Indeed, each retained edge outside $Q$ with both ends in $X$ belongs to the fixed $p$-tree and is directed away from $p$. Each retained edge outside $Q$ with both ends in $Y$ belongs to the fixed $q$-tree and is directed towards $q$. Every noncritical crossing edge is directed from $X$ to $Y$, and every retained edge of $Q$ is directed towards $p$. Thus each edge receives the same direction in every cycle containing it.

There is at least one noncritical crossing edge. Otherwise deleting the critical edge would leave no edge between the nonempty sets $X$ and $Y$, making it a bridge of $G$. Let $H$ consist of the edges and vertices of all the constructed cycles, with the directions just specified. Only the tree segments used in these cycles are retained. Every cycle contains the critical edge, so their union is connected. Also, every edge of $H$ belongs to a cycle, and hence $H$ is bridgeless. To see that its orientation is strong, take any vertices $u,v\in V(H)$ and choose a constructed cycle through each of them. The first cycle gives a directed walk from $u$ to $Q_c$, and the second gives a directed walk from $Q_c$ to $v$. Their concatenation can be shortened to a directed $u$--$v$ path.

Each noncritical crossing edge belongs to its own constructed cycle, and the critical edge belongs to every such cycle. Thus $H$ contains all crossing edges, and its only edge directed from $Y$ to $X$ is the critical edge. Its intersection with $Q$ is the union of the segments $Q_\lambda\cdots Q_\nu$. Since each segment contains $Q_cQ_{c+1}$, this union is a single subpath. Denote its length by $h$. Then $1\le h\le\rho$.

We next prove \eqref{eq:core-distance}. For a crossing edge $xy$, where $x\in X$, write $m=a(x)+1+b(y)-\rho$. A shortest $p$--$x$ path, the edge $xy$, and a shortest $y$--$q$ path form a $p$--$q$ walk of length $a(x)+1+b(y)$. Hence $m\ge0$. If $d$ is odd, then $y\in Y$ gives $b(y)\le a(y)-\theta-1\le\rho-\theta-1=d-1$, so $m\le d$. If $d$ is even, then $x\in X$ gives $a(x)\le b(x)+\theta\le d+\theta=\rho-2$, so $m\le d-1$. In both cases, $0\le m\le d$.

For each arc, express its length one as its increase in $\varphi$ plus a correction. A tree arc outside $Q$ with both ends in $X$ increases $a$ by one, and one with both ends in $Y$ decreases $b$ by one; their corrections are therefore zero. A noncritical crossing arc $x\to y$ has $\varphi(y)-\varphi(x)=d+1-m$, so its correction is $m-d\le0$. On $Q$, $\varphi(Q_i)=i$ for $i\le c$ and $\varphi(Q_i)=d+i$ for $i\ge c+1$. Thus each reversed edge of $Q$ within one part has correction $2$, whereas $Q_{c+1}\to Q_c$ has correction $d+2$.

For distinct $u,v$, choose a directed $u$--$v$ path in $\rbjt{H}$ with no repeated vertex. It uses at most $h-1$ reversed edges of $Q$ within the parts and uses the reversed critical edge at most once. All its other corrections are nonpositive. Summing the increases in $\varphi$ along this path leaves $\varphi(v)-\varphi(u)$, while the sum of its corrections is at most $2(h-1)+(d+2)=d+2h$. Its length is therefore at most the right-hand side of \eqref{eq:core-distance}. The path of length zero proves the same inequality when $u=v$.

Finally, if a component of $G-V(H)$ met both $X$ and $Y$, a path in that component would contain a crossing edge. This is impossible because every crossing edge belongs to $H$. Thus each outside component lies in one part. An attachment in the other part would also give a crossing edge with an endpoint outside $H$, which is again impossible. Each component and all its attachments therefore lie in the same part.
\end{proof}

Every vertex outside $H$ is at distance at most $r-1+\varepsilon$ from $H$ if it lies in $X$, and at most $r$ if it lies in $Y$. Indeed, for $v\in X\setminus V(H)$, a shortest $v$--$q$ path meets $H$, since $q\in Y$ and every crossing edge belongs to $H$. Let $s$ be its first vertex in $H$. The segment from $v$ to $s$ contains no crossing edge, since such an edge would have an endpoint outside $H$, so $s\in X$. Its length is $b(v)-b(s)$, because the remaining segment to $q$ is also shortest. As $b(v)\le d$ and $b(s)\ge r+1$, this length is at most $d-r-1=r-1+\varepsilon$. Similarly, for $v\in Y\setminus V(H)$, a shortest $v$--$p$ path meets $H$, since $p\in X$ and every crossing edge belongs to $H$. Its first vertex $s$ in $H$ lies in $Y$ by the same argument. The segment from $v$ to $s$ has length $a(v)-a(s)$, because the remaining segment to $p$ is shortest, and this is at most $\rho-t-\varepsilon=r$, since $a(v)\le\rho$ and $a(s)\ge t+\varepsilon$. These bounds explain why the auxiliary root eccentricities in the next section are at most $\lfloor d/2\rfloor$.

The reversed crossing edge accounts for the term \(d\) in \eqref{eq:core-distance}. For attachments nearest the boundary, the selected tree paths remove this term.

\begin{lemma}\label{lem:shallow-attachments}
Let $H$ have the fixed strong orientation obtained from the directed cycles through $Q_{c+1}\to Q_c$. The additional term \(d+2h\) in \eqref{eq:core-distance} can be replaced by \(2h+\varepsilon\) for an \(X\) source \(u\) with \(b(u)=r+1\), and by \(2h+1-\varepsilon\) for a \(Y\) target \(v\) with \(a(v)=c+1\). If both conditions hold and \((u,v)\ne(Q_c,Q_{c+1})\), their actual directed distance is at most \(2h+1\). For every noncritical crossing edge with the value \(m\), we also have \(\dist_{\rbjt{H}}(Q_c,Q_{c+1})\le2h-1+m\).
\end{lemma}

\begin{proof}
For such an \(X\) source, the partition inequalities give \(c\le a(u)\le c+\varepsilon\). Its selected \(q\)-tree path immediately crosses to \(Y\), where it stays. If \(u\ne Q_c\), that first edge is noncritical, so its defining cycle retains the whole segment from \(u\) to its first vertex \(Q_i\) on $Q$, with \(i\ge c+1\). The segment is unchanged and has length \(i-c\). Prepending it to \eqref{eq:core-distance} from \(Q_i\), where \(\varphi(Q_i)=d+i\), gives \((i-c)+\varphi(v)-(d+i)+d+2h\le\varphi(v)-\varphi(u)+2h+\varepsilon\). If \(u=Q_c\), a simple directed path starting there avoids the reversed critical edge, so the correction sum in Lemma~\ref{lem:central-core} is at most \(2(h-1)\).

For a \(Y\) target \(v\) with \(a(v)=c+1\), put \(\chi=b(v)-r\), where \(0\le\chi\le1-\varepsilon\). Every earlier vertex of its selected \(p\)-tree path has \(a\le c\), hence lies in \(X\). If \(v\ne Q_{c+1}\), the last edge is noncritical, and its defining cycle retains the segment from its last vertex \(Q_j\) on $Q$ to \(v\), with \(j\le c\). Its length is \(c+1-j=\varphi(v)-\varphi(Q_j)-d+\chi\). Appending this segment to \eqref{eq:core-distance} proves the target bound. For \(v=Q_{c+1}\), a simple path ending there avoids the reversed critical edge.

When both conditions hold, concatenate these two segments with the reversed \(Q_i\)--\(Q_j\) segment. Take \(i=c\) when \(u=Q_c\), and \(j=c+1\) when \(v=Q_{c+1}\). Except for the excluded pair, \(i\ge j\), and the total length is \((i-c)+(i-j)+(c+1-j)=2(i-j)+1\le2h+1\). Finally, let the selected paths of a noncritical edge \(xy\) leave and meet $Q$ at \(Q_\lambda,Q_\nu\). The segment through this edge, whose internal vertices lie outside $Q$, has length \(\nu-\lambda+m\). From \(Q_c\), follow $Q$ backwards to \(Q_\lambda\), use that segment, and follow $Q$ backwards from \(Q_\nu\) to \(Q_{c+1}\). The resulting directed walk has length \(2(\nu-\lambda)+m-1\le2h+m-1\). Deleting closed subwalks completes the proof.
\end{proof}

Lemma~\ref{lem:shallow-attachments} leaves one exceptional pair, namely \(Q_c,Q_{c+1}\). The next lemma chooses $Q$ so that a short route also joins this pair. Its only alternative is handled immediately afterward. The covering radius of an edge \(uv\) is \(\max_x\min\{\dist(u,x),\dist(v,x)\}\).

\begin{lemma}\label{lem:axis-choice}
If \(t\ge1\), the centre \(p\), its farthest vertex \(q\), and the shortest path \(Q\) can be chosen so that a noncritical crossing edge has a value of \(m\) at most \(r+\varepsilon\). Consequently, \(\dist_{\rbjt{H}}(Q_c,Q_{c+1})\le2h+r-1+\varepsilon\).
If \(d\) is odd and \(t=0\), either the same choice with a value of \(m\) at most \(r+1\) is possible, or some edge has covering radius at most \(r\).
\end{lemma}

\begin{proof}
Start with any \(p,q,Q\) and put \(u=Q_1\). We find, allowing this triple to change, a vertex \(w\ne Q_c\) with \(b(w)\le r+\varepsilon\) and a \(p\)--\(w\) geodesic avoiding the critical edge.

First suppose a farthest vertex \(z\) of \(p\) satisfies \(\dist(u,z)\ge\rho\). A \(p\)--\(z\) geodesic \(Z\) meets \(Q\) only at \(p\): an intersection \(Z_i=Q_i\), \(i\ge1\), would give a \(u\)--\(z\) route of length \(\rho-1\). A shortest \(q\)--\(z\) path avoids \(Q_c,Z_c\), since at either vertex the sum of the distances to its endpoints is at least \(2(\rho-c)=2r+2>d\). Its midpoint \(w\) is within \(r+\varepsilon\) of both endpoints. A \(p\)--\(w\) geodesic cannot use both critical edges, since both join distance layers \(c,c+1\). Keep the farthest vertex and the corresponding shortest path whose critical edge it avoids.

Otherwise every farthest vertex of \(p\) is at distance \(\rho-1\) from \(u\). All remaining vertices are within \(\rho\) of \(u\), so \(u\) is also a centre. Choose a farthest vertex \(z\) of \(u\). Then \(\dist(p,z)=\rho-1\), and a \(u\)--\(z\) geodesic \(Z\) can start \(u,p\). If \(\varepsilon=1,t=0\), every vertex is within \(r\) of \(p\) or \(u\), giving the covering alternative. Hence assume \(t\ge1\). A shortest \(q\)--\(z\) path again avoids \(Q_c,Z_c\): for \(c\ge1\), the endpoint distance sum at either vertex is at least \(2r+3\); for \(c=0\), it is \(2\rho-1>d\). Take its midpoint \(w\), and geodesics to \(w\) from \(p\) and \(u\). If the first uses \(Q_cQ_{c+1}\), replacing its prefix by the \(u\)--\(Q_{c+1}\) segment gives \(\dist(u,w)\le\dist(p,w)-1\). If the second uses \(Z_cZ_{c+1}\), the reverse inequality with an additional \(-1\) follows similarly. They cannot both hold. Keep the successful centre, farthest vertex, and shortest path.

Rename the resulting objects \(p,q,Q,w\). If \(w\in Y\), the first crossing edge \(xy\) on the chosen \(p\)--\(w\) geodesic is noncritical and has \(m=a(y)+b(y)-\rho\le a(w)+b(w)-\rho\le b(w)\le r+\varepsilon\), because \(a+b\) is nondecreasing along that geodesic and \(p\) is a centre. If \(w\in X\), then \(\varepsilon=1\), \(b(w)=r+1\), and \(c\le a(w)\le c+1\). The first edge of a shortest \(w\)--\(q\) path crosses to \(Y\), has \(m=a(w)-c\le1\), and is noncritical since \(w\ne Q_c\). Lemma~\ref{lem:shallow-attachments} gives the claimed distance bound.
\end{proof}

The alternative concerning the covering radius of an edge in Lemma~\ref{lem:axis-choice} needs no further path construction. The following bound settles it by orienting a short cycle and using the quoted radius theorem.

\begin{lemma}\label{lem:edge-cover}
If an edge has covering radius at most \(R\), then \(G\) has a strong orientation of diameter at most \(2R^2+4R+1\).
\end{lemma}

\begin{proof}
For such an edge \(uv\), take a shortest \(u\)--\(v\) path in \(G-uv\), of length \(m\). Its \(i\)th vertex has distance \(\min\{i,m-i\}\) from \(\{u,v\}\): a shorter route up to the first endpoint it meets avoids \(uv\), and would shorten a prefix or suffix of the chosen path. Thus \(\lfloor m/2\rfloor\le R\), and this path with \(uv\) forms a cycle of length at most \(2R+2\). Orient the cycle cyclically and contract it to a root, deleting loops. The quotient is connected and bridgeless, and its root eccentricity is at most \(R\). Theorem~\ref{thm:radius} bounds both directed root distances by \(R(R+1)\). Restore their actual endpoints on the cycle and join them using at most \(2R+1\) arcs. Every ordered pair is joined by a directed walk of length at most \(2R(R+1)+2R+1\), as required.
\end{proof}

\section{Rooted orientations with recorded attachments}\label{sec:rooted-orientations}
This section supplies directed paths from the central subgraph to every outside vertex and back. The important point is to record where each path meets the central subgraph, since Section~\ref{sec:joining-paths} will use the directed distance between these actual attachments. Lemma~\ref{lem:auxiliary-attachments} first builds an auxiliary rooted graph whose distances are exact. Lemma~\ref{lem:rooted-orientations} then gives the basic recursive orientations. The later lemmas improve their path lengths using the radius and diameter of the original graph.

Fix a central subgraph $H$ containing every crossing edge, together with a strong orientation of $H$. The construction below applies both to the initial subgraph and to the rebuilt subgraph in Section~\ref{sec:rebuilding}. In either case, $q$ denotes the chosen farthest vertex and $b=\dist_G(q,\cdot)$ is the true distance in $G$, even if a retained tree is not geodesic. An added path from an auxiliary root to an attachment is called an arm. The paired arms can be oriented in opposite directions. Retained nonroot vertices are called ordinary; the internal vertices of added arms are artificial, and every edge of an added arm is called artificial. A real vertex or edge belongs to the original graph $G$. The designation ordinary is relative to the current auxiliary graph, whereas real always refers to $G$. A length-one arm has an artificial edge but no artificial internal vertex. Depth means distance from the root in the current auxiliary graph.

\begin{lemma}\label{lem:auxiliary-attachments}
Let $V$ be a component of $G-V(H)$, and let $S$ be its neighbours in $H$. Retain all edges within $V$ and between $V$ and $S$. Assign $w(s)=b(s)-r-1$ on $X$ and $w(s)=a(s)-t-\varepsilon$ on $Y$. Identify the zero-weight attachments with a new root $o$. At each positive attachment $s$, add two $o$--$s$ paths of length $w(s)$ with disjoint private interiors. Record the original endpoint of every attachment edge, and delete loops created by identification.

The resulting multigraph is connected and bridgeless. The depth of an ordinary vertex $v$ is exactly $b(v)-r-1$ on $X$, and $a(v)-t-\varepsilon$ on $Y$. The root eccentricities are at most $r-1+\varepsilon$ and $r$, respectively, and every added arm is geodesic. Strong orientations of these auxiliary graphs restore to a strong orientation of $G$ extending the fixed orientation of $H$.
\end{lemma}
\begin{proof}
The partition inequalities make the weights nonnegative, and $V\cup S$ lies in one part. An $X$ vertex with $b=r+1$ is incident with a crossing edge on a shortest path to $q$; a $Y$ vertex with $a=t+\varepsilon$ has the analogous property on a shortest path to $p$. Such vertices already belong to $H$. Thus the assigned depths are positive outside $H$.

Assign each ordinary vertex its asserted depth, assign $o$ value zero, and interpolate consecutively along every arm. Values differ by at most one across each edge, giving a lower bound on root distance. Conversely, follow a shortest $v$--$q$ path on $X$, or $v$--$p$ path on $Y$, to its first attachment $s$. The preceding vertices remain in $V$, since all crossing edges lie in $H$. Prepending the arm at $s$ gives the asserted depth. This also proves that the arms are geodesic. The eccentricity bounds follow from $b\le d$ and $a\le\rho$; arm interiors are no deeper than their endpoints.

Connectivity is immediate. Each artificial edge lies on the cycle formed by its arm pair. If a retained edge were a bridge, the side away from $o$ would contain no attachment: every positive attachment has two arms to $o$, and every zero attachment is identified with $o$. That edge would then be a bridge of $G$.

A simple directed root path uses an arm only as a complete initial or final segment, since arm interiors have degree two. Delete that arm, or restore the attachment recorded on the ordinary root edge. The resulting real path joins the vertex to its own attachment in $H$, and its length decreases by exactly the deleted arm length. Every vertex can therefore reach and be reached from $H$. Different components use disjoint original edges, so their orientations combine. Any still unoriented edge with both endpoints in $H$ may be oriented arbitrarily.
\end{proof}

\begin{theorem}\label{thm:radius}
Let $J$ be a finite connected bridgeless graph, possibly with parallel edges, let $z$ be a specified vertex, and let $R$ be a nonnegative integer. If $\dist_J(z,x)\le R$ for every vertex $x$, then $J$ has a strong orientation $\rbjt{J}$ in which both $\dist_{\rbjt{J}}(z,x)$ and $\dist_{\rbjt{J}}(x,z)$ are at most $R(R+1)$ for every vertex $x$.
\end{theorem}
This form follows from Chv\'atal and Thomassen~\cite{ChvatalThomassen1978} and is stated explicitly in~\cite[Theorem~3]{DankelmannEtAl2024}. In~\cite{DankelmannEtAl2024}, it bounds oriented diameter using a connected vertex set within a prescribed distance of every vertex. Here it settles the small-diameter and minimum-radius cases. We next retain the attachment information needed when root paths are restored.

Henceforth $J$ is loopless, connected and bridgeless, rooted at $o$, with possible prescribed arm pairs. A pair consists of two $o$--$s$ paths of length $w(s)=\dist_J(o,s)\ge1$, whose private internal vertices have degree two; each endpoint has at most one pair. The arm endpoints and all nonroot vertices outside the arm interiors are ordinary, as are all edges outside the added arms. Every ordinary root edge records its attachment before identification, including when different edges become parallel. A child root may represent several attachments rather than a single vertex of $G$.

An $o$--$x$ directed path is an input, and an $x$--$o$ path a return. For ordinary $x$, a path is positive if its root segment is a complete arm, and zero if that segment is an ordinary root edge. The word zero describes the type of root segment, not the length of the path. Every path is restored at its own recorded attachment.

When a positive path is restored to the original graph, its length decreases by the length of the arm actually used. The depth bound and the bound for that arm must therefore hold on the same chosen path. The even-diameter argument also needs an exact relation between a selected child path and its restored parent path, which will be enforced by selecting the child minimum first.

We need three kinds of length bound. The bound $C_R(z)$ controls an ordinary path at depth $z$. The smaller bound $I_R(z)$ applies to the classes specified in Lemma~\ref{lem:rooted-orientations}. For a positive return, $W_R(w)$ also records the length $w$ of the arm actually used. Their largest depth bounds are $A_R=C_R(R)$ and $F_R=I_R(R)$. Define
\[
\begin{aligned}
C_R(z)&=z(2R-z+1),& I_R(z)&=(z-1)(2R-z)+1,\\
W_R(w)&=R+(R-w)^2,& A_R&=R(R+1),\qquad F_R=R^2-R+1,
\end{aligned}
\]
for $1\le z,w\le R$, with $C_R(0)=0$. The parameter $R$ may exceed the actual root eccentricity. Whenever the indices are defined, the differences from the maximum-depth bounds are $C_R(z)-A_R=I_R(z)-F_R=-(R-z)(R-z+1)$. We will use the restoration identities $C_{R-1}(z-1)+2R=C_R(z)$, $C_{R-1}(z-1)+1=I_R(z)$, and $I_{R-1}(z-1)+2R-2=I_R(z)$. The arm bound also satisfies $W_R(1)=F_R$.
The depth bounds increase with $R$ and integer $z$; their successive depth increments are $2(R-z)$. The arm bound increases with $R$ and decreases with $w$. The next lemma introduces two rooted orientation constructions, denoted by $\mathcal C$ and $\mathcal H$; the ordinary letter $H$ continues to denote the central subgraph of $G$.

\begin{lemma}\label{lem:rooted-orientations}
There are rooted orientation constructions $\mathcal C$ and $\mathcal H$ with the following properties whenever $\ecc_J(o)\le R$. Each gives a strong orientation with a chosen input and return of length at most $C_R(z)$ for every vertex of depth $z$. A chosen positive ordinary return using an actual arm of length $w$ also satisfies $I_R(z)$ and $W_R(w)$ on that same path.

The tree and colouring constructed below have the following properties. Every ordinary vertex having a root geodesic that begins with a complete arm is red, and its tree path begins with such an arm. Red ordinary vertices whose tree paths begin with complete arms satisfy $I_R(z)$ in both directions, as do ordinary vertices of depth at least two in the subgraph $A$ defined in the proof. Returns from blue ordinary vertices, and inputs in $\mathcal C$ to the other red vertices, also satisfy this bound. Returns from these other red vertices and from blue ordinary vertices are zero.

For each construction, these chosen paths are included among the candidate paths in its resulting fixed orientation. Reversing all edges exchanges inputs and returns.
\end{lemma}
\begin{proof}
The construction has four steps. We first consider the components obtained by deleting the root. In the remaining case, we choose a shortest-path tree and orient a subgraph $A$ through the root. We next construct smaller rooted graphs for the components outside $A$. Finally, we restore the child paths and verify their lengths and attachment records.

Construct both orientations jointly by induction on the actual root eccentricity and, at fixed eccentricity, on the number of vertices. Monotonicity then permits a larger prescribed $R$. The one-vertex case is immediate.

\emph{Deleting the root.} If $J-o$ is disconnected, each component together with $o$ is bridgeless: a simple cycle through $o$ stays in one component. Each arm pair lies in one such subgraph, whose root distances are unchanged and whose vertex number is smaller. Combine the separate constructions.

If $J-o$ is connected and $o$ has only one distinct neighbour $v$, every cycle through $o$ is a pair of parallel $ov$ edges. Thus $J-o$ is bridgeless, its depths from $v$ are the parent depths minus one, and any prescribed arms have length one. Both constructions apply $\mathcal H$ to this graph with no prescribed arms. Orient two $ov$ edges oppositely, using the prescribed pair if present, and append them to the child paths. Their lengths are at most $C_{R-1}(z-1)+1=I_R(z)$; a prescribed return also satisfies $A_{R-1}+1=W_R(1)$. The point $v$ itself uses the two root edges. Colour all ordinary vertices red if a pair is present and blue otherwise; in the red case their tree paths begin with the prescribed arm. Orient any remaining root edges arbitrarily.

\emph{The tree and subgraph $A$.} Assume that $J-o$ is connected and $o$ has at least two distinct neighbours. Choose a main and a secondary arm in each pair. Build a shortest-path tree by assigning each arm interior its preceding arm vertex, each arm endpoint its main-arm route, and the remaining parents in increasing depth, preferring a main-arm route whenever possible. Replacing an initial secondary arm by its main arm, then following increasing depths, shows that every ordinary vertex with a positive geodesic receives a main-arm tree route.

Label each tree branch by its first vertex after $o$, and join two labels when their branches are adjacent in $J-o$. This label graph is connected. The branch beginning with a secondary arm contains only its internal vertices and neighbours only the label of its paired main arm; a length-one pair has no separate branch for the secondary arm. Colour the labels of main arms red and, when a separate branch exists, the labels of secondary arms blue. Colour every still uncoloured neighbour of the label of a main arm blue. In the subgraph induced by the remaining uncoloured labels, choose a maximal independent set, colour its labels red, and colour the rest blue. The other red branches contain no arm. Vertices inherit the colour of their branch, independently of $X$ and $Y$.

The red labels not belonging to main arms are pairwise nonadjacent and have no neighbour belonging to a main arm. Every blue label outside the branches containing secondary arm interiors has a red neighbour by construction, and each of the other red labels has such a blue neighbour by connectivity. Each label of a main arm has a blue neighbour unless its length-one pair already supplies the cycle used below.

Let the subgraph $A$ contain all red--blue edges, their endpoint tree routes to $o$, and every length-one arm pair. Orient red tree edges away from $o$, blue tree edges towards $o$, and red--blue edges from red to blue; orient each length-one pair oppositely. The resulting directed cycles through $o$ make $A$ strong. The label properties put every root neighbour in $A$. All arms also lie in $A$: for a longer pair, the edge from its endpoint to its last secondary interior vertex is red--blue, and its supporting routes include the pair. Include and arbitrarily orient any remaining edge with both endpoints in $A$.

At depth $j$ in $A$, the tree direction has length $j$: it is an input for a red vertex and a return for a blue vertex. For a vertex on an endpoint tree route of a red--blue edge, the opposite root direction follows the rest of the corresponding directed cycle: the segment in its own branch has length at most $R-j$, the crossing edge has length one, and the other tree segment has length at most $R$. Their total is at most $(R-j)+1+R=2R+1-j$. A vertex included only through a length-one arm pair is its ordinary endpoint and instead uses the two oppositely oriented arms, with length one in each root direction. 

Since $1\le j\le R$, we have $j\le2R+1-j$ and $C_R(j)-(2R+1-j)=(j-1)(2R-j+1)\ge0$, so both lengths are at most $C_R(j)$. If $j\ge2$, then $I_R(j)-(2R+1-j)=(j-2)(2R-j)\ge0$ and $I_R(j)-j=(j-1)(2R-j-1)\ge0$, so both lengths are at most $I_R(j)$.
Inputs to the other red vertices and returns from blue ordinary vertices follow the tree and satisfy $I_R(j)$ also when $j=1$. An arm endpoint uses its two arms, of length $w\le I_R(w),W_R(w)$.

Fix the returns in $A$ as follows. A red ordinary vertex whose tree path begins with a complete arm, and which lies strictly beyond its arm endpoint, has depth at least two. A supporting red--blue edge for this vertex meets a blue branch containing no secondary arm interiors, because these interiors meet only their own endpoint. Its return therefore ends with an ordinary root edge in a blue branch. The same holds for any other red vertex in $A$, whose label neighbours neither the label of a main arm nor that of a secondary arm. Keep these fixed routes in $A$ available.

\emph{Child graphs.} We use two rules for constructing a rooted child graph from an outside component. The first preserves its parent depths minus one; the second allows shorter routes through attachments of parent depth at least two. The constructions \(\mathcal C\) and \(\mathcal H\) differ only in which rule they use on red components in branches containing no arm.

For each component $U$ of $J-V(A)$, retain its internal edges and all edges to its neighbours $S$ in $A$. Every red--blue edge and every artificial interior lies in $A$, so $U\cup S$ has one colour and contains no artificial interior. A red component lies either in branches containing main arms or in one other red branch; a blue component lies only in branches containing no secondary arm interiors.

For the first rule, identify depth-one attachments with the child root, and give every attachment of parent depth $j\ge2$ two private arms of length $j-1$. Each retained vertex of parent depth $z$ has child depth exactly $z-1$: the assigned values ``parent depth minus one'', interpolated along arms, give the lower bound; its parent tree route to its first attachment gives equality. The child eccentricity is at most $R-1$.

For the second rule, used only for a red component in a branch containing no arm and labelled $v$, identify its possible unique depth-one attachment $\alpha=v$ with the child root, and give every other attachment two private length-one arms. If $\alpha$ is absent, give all attachments such arms. Parent tree routes give child depth at most $z-1$ and eccentricity at most $R-1$. Every ordinary child root edge records $\alpha$. In both child graphs the added arms are geodesic: this follows from the exact depths under the first rule, and is immediate for the length-one arms under the second rule.

The children are connected and bridgeless. Artificial edges lie on paired-arm cycles. If a retained edge were a bridge, its side away from the child root would contain no attachment and would certify a bridge of $J$.

Both constructions use the first rule, with $\mathcal H$ on red components in branches containing main arms and reversed $\mathcal H$ on blue components. Here reversed $\mathcal H$ means that the entire child orientation is reversed before restoration, so a parent input on a blue component comes from a return in the unreversed $\mathcal H$ child orientation. On red components in branches containing no arm, $\mathcal C$ uses the first rule and $\mathcal H$ the second rule, both with $\mathcal H$ children. All calls have smaller actual root eccentricity.

\emph{Restoring paths.} Delete the actual child arm, or restore its recorded zero attachment, and append the fixed path in the parent subgraph $A$ at that attachment. At an attachment of depth $j$, the first rule has net increase one in the tree direction and at most $2R+2-2j$ in the opposite direction. Under the second rule, for $j\ge2$ these increases are at most $j-1\le R-1$ and $2R-j\le2R-2$; at $j=1$ they are one and at most $2R$. Thus every direction satisfies $C_{R-1}(z-1)+2R=C_R(z)$.

Returns from blue ordinary vertices, inputs in $\mathcal C$ to the other red vertices, and inputs to red ordinary vertices whose tree paths begin with complete arms use the tree direction under the first rule, giving $C_{R-1}(z-1)+1=I_R(z)$. For a return from a red component in branches containing main arms, a zero child return records a parent attachment of depth one. This is the first vertex of a main arm. If that arm had length greater than one, this vertex would be an artificial interior of degree two with both neighbours in $A$, and could not neighbour $U$. Thus it is an actual parent-arm endpoint of length one, and restoration adds one. A positive child return uses a child arm of length $j-1$ at a parent attachment of depth $j\ge2$. The induction hypothesis gives the $I_{R-1}(z-1)$ bound on this same child return, and its net restoration is at most $(2R+1-j)-(j-1)\le2R-2$. Both cases therefore give $I_R(z)$.

If the child return from another red vertex under the second rule is positive, it uses a length-one arm at a parent attachment of depth $j\ge2$. Deleting that arm and appending the fixed return in $A$, of length at most $2R+1-j$, gives a net increase at most $2R-j\le2R-2$. Its child depth $z'$ satisfies $z'\le z-1$, so the induction hypothesis and monotonicity give a parent length at most $I_{R-1}(z')+2R-2\le I_{R-1}(z-1)+2R-2=I_R(z)$. A zero child return ends at $\alpha$ and uses its fixed return in $A$ ending with an ordinary root edge in a blue branch. Under the first rule, a return from another red vertex also uses a fixed route in $A$ ending at such an edge. A return from a blue ordinary vertex uses its blue tree route. These parent returns are therefore zero.

A positive parent return occurs only at an arm endpoint or in a red component in branches containing main arms. At an endpoint its length is $w$. A zero child return in such a red component restores a parent arm of length one, and satisfies $A_{R-1}+1=W_R(1)$. If both child and parent returns are positive, their actual arms end at the same vertex, the child arm has length $w-1$, and the net increase is exactly one. Since $W_{R-1}(w-1)+1=W_R(w)$, the $I$ bound and the bound for the actual arm hold on the same path.

Join the child orientations to the orientation of $A$. Every vertex has both root directions, so the orientation is strong. A restored walk can be shortened by deleting closed subwalks without losing its root edge: the root occurs only at the appropriate endpoint. Its parallel-edge identity and complete arm survive as well, because arm interiors occur nowhere in the ordinary part. Thus the actual attachment record is preserved. At actual eccentricity one there are no external components, so the calculation in $A$ completes the induction.
\end{proof}

The next lemma explains the purpose of the second rule. A path that begins with an arm and is only \(\sigma\) longer than a geodesic yields a return with only \(2\sigma\) added to the smaller depth bound. The path need not already follow the orientation.

\begin{lemma}\label{lem:positive-root-path}
Let $x$ be an ordinary vertex of depth $z$ in the fixed $\mathcal H$ orientation. If a simple undirected root--$x$ path of length $z+\sigma$ begins with a complete arm, there is a return of length at most $I_R(z)+2\sigma$. Every new return in the proof is zero; the original positive returns retain their pointwise and actual-arm bounds.
\end{lemma}
\begin{proof}
When $\sigma=0$, the tree-parent choice preferring a main-arm route makes $x$ red and gives it a tree path beginning with a complete arm, so Lemma~\ref{lem:rooted-orientations} supplies the required bound. Suppose $\sigma\ge1$. Red ordinary vertices whose tree paths begin with complete arms, blue ordinary vertices, and ordinary vertices in $A$ of depth at least two already have the $I_R(z)$ return bound. Thus only another red vertex outside $A$, or another red vertex in $A$ of depth one, remains. We first consider the external case, and write $v$ for its branch label.

Let $a$ be the given path's last vertex in $A$, of depth $j$, and write the length of its final segment as $z-j+\tau$, where $\tau\ge0$. The starting arm ends in a branch containing a main arm, so the path must enter the branch labelled $v$ through a blue branch containing no secondary arm interiors. Let $bc$ be its last blue--red entry, with depths $d_b,d_c$. Write $L$ for the length of the $c$--$a$ segment and $d_b+e$ for that of the root--$b$ prefix. Then $e\ge0$ and $\sigma=e+\tau+d_b+1+L-j$. Since $d_b+1+L\ge j$, we also have $\tau\le\sigma$. Since $d_c\le j+L$, the fixed return in $A$ along the red tree from $v$ to $c$, across $cb$, and down the blue tree has length $d_c+d_b\le2j+\sigma-\tau-1$. It ends with an ordinary root edge: a secondary arm interior cannot neighbour a red branch containing no arm.

If $j\le\tau+1$, this continuation has length at most $2\sigma+1$. A zero child return ends at $v$, so appending it gives length at most $C_{R-1}(z-1)+2\sigma+1=I_R(z)+2\sigma$. A positive child return already restores to $I_R(z)$ by Lemma~\ref{lem:rooted-orientations}.

If $j\ge\tau+2$, the length-one child arm at $a$ followed by the final segment has length $1+z-j+\tau\le z-1$. Every zero child root path has length at least $z-1$: appending its recorded parent root edge would otherwise contradict the parent depth $z$. Thus $x$ has a positive child geodesic. If its child depth is less than $z-1$, every child geodesic is positive; if it equals $z-1$, the root path through the arm at $a$ is geodesic. The tree-parent choice preferring a main-arm route makes $x$ red in the child and gives it a tree path beginning with a complete arm. Its $I$ return restores to at most $I_{R-1}(z-1)+2R=I_R(z)+2\le I_R(z)+2\sigma$.

If $x=v$ is another red vertex in $A$ of depth one, the same notation gives $\sigma=d_b+e+L$ and $d_c\le1+L$. Its fixed return in $A$ therefore has length at most $\sigma+1\le I_R(1)+2\sigma$. If deleting the root disconnects the graph, the simple path lies in one component together with the root, and we restrict to that subgraph. The one-neighbour case already has the $I$ bound, and smaller actual eccentricity is handled by monotonicity. Every new continuation ends with an ordinary root edge in a blue branch; shortening preserves it by Lemma~\ref{lem:rooted-orientations}.
\end{proof}

\subsection{Using the original graph}

The basic bounds depend only on the auxiliary root eccentricity. We now use the original radius and diameter as well. A first branch vertex must be far from some vertex of the original graph; a shortest path to that vertex forces an exit from the current component. Lemma~\ref{lem:positive-root-path} turns this exit into a shorter directed return. To use the argument recursively, every original exit must remain recorded.

Suppose that the ordinary vertices are retained from a connected graph $G$ of diameter $D$ and radius $\rho$. At every recursive step, every original edge leaving an ordinary component is recorded as an attachment edge, and every ordinary root edge retains the identity of its endpoint before identification. These two requirements persist under Lemma~\ref{lem:auxiliary-attachments} and both rules for constructing the child graph in Lemma~\ref{lem:rooted-orientations}. An older positive attachment lies in the new subgraph $A$ through its arms; a neighbour of an older zero attachment lies there because every root neighbour is in $A$. Neither can enter a new external component. Every other original exit is retained as a new attachment, and artificial interiors have no original incident edges.

\begin{lemma}\label{lem:ambient-savings}
Assume the two requirements on retained edges and their original endpoints stated above, $R<\rho$, and $0\le\lambda=\rho+R-D\le R$. In either $\mathcal H$ or $\mathcal C$, both ordinary root directions have length at most $A_R-g$, where $g=(2\lambda-7)_+$. In $\mathcal H$, an ordinary vertex of depth $z$ has a return of length at most $C_R(z)-\bigl(2(\rho+z-D)-5\bigr)_+$. A selected positive return with actual arm length $w$ satisfies $I_R(z)$, $W_R(w)$ and $F_R-g$ on the same path. These bounds hold in the same fixed recursive orientations.
\end{lemma}

\begin{proof}
We first improve returns in \(\mathcal H\) at each depth. Write $Q(x)$ for the length of a return selected after these improvements. Here $Q(x)$ is a numerical return length in the current rooted graph; $Q$ alone still denotes the fixed geodesic in $G$. We then propagate this improvement through the two recursive constructions, checking the bounds for positive returns on the same restored paths.

There is nothing to prove when $R=0$. For the pointwise $\mathcal H$ bound, write $i=R-z$. The required saving $(2\lambda-2i-5)_+$ is at most $(2z-5)_+$. Thus a smaller actual eccentricity, which gives $C_{R-1}(z)=C_R(z)-2z$, already suffices, as does any existing $I_R(z)$ return. Restriction to a component of $J-o$ together with $o$ preserves root distances, and depth one requires no saving. It remains to consider a red component $V$ outside $A$ in a branch containing no arm and a vertex $x\in V$.

Let $v$ be the first vertex of its tree branch. The tree routes from $v$ to $V$ and its attachments are real: a depth-$h$ attachment is within distance $h-1$ of $v$, and every vertex of $V$ is within distance $R-1$. Choose $Z\in V(G)$ with $\dist_G(v,Z)\ge\rho$. Since $R<\rho$, the vertex $Z$ is outside $V$. A shortest $x$--$Z$ path first leaves $V$ at a recorded attachment $a$. If $a$ has depth $h$ and the initial $x$--$a$ segment has length $m$, then $z-h\le m\le D-\rho+h-1=R-\lambda+h-1$, since the remaining segment has length at least $\rho-h+1$. If the required saving is positive, then $\lambda-i\ge3$, and these inequalities exclude $h=1$. The child constructed by the second rule therefore has a length-one arm at $a$.

Let its actual eccentricity be $R'\le R-1$ and the depth of $x$ be $z'\le z-1$. The arm followed by the real $a$--$x$ segment is a simple root path of length $1+m$. Lemma~\ref{lem:positive-root-path} gives a child return of length at most $I_{R'}(z')+2(1+m-z')$. This expression is nondecreasing in the permitted integer eccentricity and depth: each depth increment in $I$ is at least two, while the other term decreases by two. Increase the eccentricity bound to $R-1$ and then the depth to $z-1$. Restoration at the actual attachment costs at most $2R$, so
\[
\begin{aligned}
Q(x)&\le I_R(z)+2+2(m-z+2)\\
&\le I_R(z)+2(R-\lambda+i)+4=C_R(z)-(2\lambda-2i-5).
\end{aligned}
\]
For zero required saving the original pointwise bound suffices. The new returns are zero returns from the other red vertices, so the original positive returns remain available. Since $i(i+1)\ge2i$, the same calculation also gives
\begin{equation}\label{eq:uniform-return}
Q(x)\le A_R-(2\lambda-5)_+\qquad\text{in }\mathcal H.
\end{equation}

We prove the common ordinary bounds and the positive-return assertions together by the induction of Lemma~\ref{lem:rooted-orientations}. A child with prescribed bound $R-1$ has parameter $\lambda-1$; if this parameter is negative, use the original bound, as the required saving is zero. A return in $\mathcal C$ from another red vertex, or an input to a blue ordinary vertex in either construction, restores an $\mathcal H$ return from a child graph constructed by the first rule. By \eqref{eq:uniform-return} its length is at most $A_{R-1}-g+2R=A_R-g$. Returns in $\mathcal H$ from the other red vertices already satisfy the stronger bound above. An input in $\mathcal H$ to another red vertex uses a child graph constructed by the second rule and has length at most $A_{R-1}-(2\lambda-9)_++R-1=A_R-\bigl[R+1+(2\lambda-9)_+\bigr]\le A_R-g$, since the positive parts differ by at most two. The remaining tree-direction restorations have length at most $I_R(z)\le F_R\le A_R-g$; returns from red ordinary vertices whose tree paths begin with complete arms retain their $I$ bound by the actual-type calculation below. Paths in $A$ have length at most $2R$. When $g=0$ use their original pointwise bounds; otherwise $R\ge4$ and $A_R-g\ge F_R+6\ge2R$.

For a return from a red ordinary vertex whose tree path begins with a complete arm, a zero child return restores through an actual parent-arm endpoint of depth one. The argument above shows that this parent arm has length one, so restoration adds one. On this same restored path, the child pointwise bound gives $I_R(z)$, its general bound gives $W_R(1)=A_{R-1}+1$, and \eqref{eq:uniform-return} gives $F_R-g=A_{R-1}-g+1$.
A positive child return retains its $I$ bound after a net increase of at most $2R-2$. If the restored parent return is also positive with actual arm $w$, then $w\ge2$, its actual child arm has length $w-1$, and the net increase is exactly one. Hence it satisfies $W_R(w)\le W_R(2)=F_R-2R+3\le F_R-g$, because $g\le(2R-7)_+$. A positive return in $A$ has length $w\le R\le F_R-g$, since $(R-1)^2\ge(2R-7)_+$, and keeps its $I,W$ bounds. In the one-distinct-neighbour case, append the prescribed edge to an improved $\mathcal H$ child return; the same length-one calculation applies. The neighbour itself has length one, and combining the orientations on the components of $J-o$ together with $o$ proves the same conclusions.

Finally, a smaller actual eccentricity $R'<R$ gives $A_{R'}\le A_R-2R\le A_R-g$ and $F_{R'}\le F_R-2R+2\le F_R-g$. Monotonicity preserves the pointwise and actual-arm bounds. All assertions therefore hold together in the fixed orientation.
\end{proof}

The following lemma applies the same path improvement to a real path from a positive attachment. In Section~\ref{sec:rebuilding}, it will convert a long minimum return into a lower bound on the distance to that attachment.

\begin{lemma}\label{lem:positive-attachment-path}
Let $x$ be an ordinary vertex of depth $z$ in a fixed $\mathcal H$ orientation with root eccentricity at most $R$. Suppose a real positive attachment $e$ is joined to $x$ by a simple path of length $l$ in the retained real graph, avoiding the auxiliary root and artificial interiors. After adding the returns constructed below to the candidate returns already obtained, their minimum $Q(x)$ satisfies $Q(x)\le I_R(z)+2+2(l-z)_+$. All added returns are zero paths; the original positive returns and their simultaneous bounds remain available.
\end{lemma}

\begin{proof}
Vertices already having the $I_R(z)$ bound need no new route. Suppose $x$ lies in a red component $V$ outside $A$ in a branch containing no arm. Let $u$ be the last vertex of $A$ on the given path, of depth $j$, and let the remaining length be $m$. The positive attachment $e$ lies in a branch containing a main arm, whereas $u$ lies in the other red branch containing $V$. The labels of these branches are distinct and nonadjacent, so the preceding real segment has length at least two; hence $m\le l-2$.

If $j\ge2$, the child constructed by the second rule has a length-one arm at $u$. Together with the last segment it gives a positive root path of length $1+m$. For child eccentricity $R'\le R-1$ and depth $z'\le z-1$, Lemma~\ref{lem:positive-root-path} bounds a child return by $I_{R'}(z')+2(1+m-z')_+$. This expression is nondecreasing in the permitted integer eccentricity and depth: a depth increment increases $I_{R'}$ by at least two and reduces the positive-part term by at most two. Increase the eccentricity to $R-1$, then the depth to $z-1$, and restore at the return's own attachment with increase at most $2R$. Since $I_{R-1}(z-1)+2R=I_R(z)+2$, this gives $Q(x)\le I_R(z)+2+2(m-z+2)_+\le I_R(z)+2+2(l-z)_+$.

If $j=1$, $u$ is the first vertex of the red branch containing no arm. Let $\nu$ be the length of the $e$--$u$ prefix. The last segment has length at least $z-1$, so $\nu\le l-z+1$. Use the last blue-to-red crossing of this prefix. The suffix from its red endpoint to $u$ has length at most $\nu-2$, since $e$ lies in a branch containing a main arm. As $u$ has depth one, the red endpoint has depth at most $\nu-1$ and its blue neighbour at most $\nu$; the fixed red-tree, crossing-edge, blue-tree route from $u$ to the root is a zero return of length at most $2\nu-1$. A zero child return concatenates at this same $u$, giving $Q(x)\le C_{R-1}(z-1)+2\nu-1=I_R(z)+2\nu-2\le I_R(z)+2(l-z)$. A positive child return already restores within $I_R(z)$. Another red vertex in $A$ of depth one is handled by the same route in $A$.

If deleting the root disconnects the graph, restrict to the component together with the root that contains the given path. Monotonicity handles smaller actual eccentricity. Every new route ends with an ordinary root edge, whose attachment record is preserved when closed subwalks are deleted.
\end{proof}

We now specify how paths are selected, since Section~\ref{sec:rebuilding} compares selected parent and child paths. Use finite lists of candidate inputs and returns supplied by the preceding proofs. At every recursive call, include all proved improvements applicable to that call in its fixed orientation, together with all candidates already obtained. Lemmas~\ref{lem:positive-root-path} and~\ref{lem:positive-attachment-path} apply directly to returns in $\mathcal H$ calls; the latter contributes the routes for every real path satisfying its hypotheses. Their improvements enter the recursive $\mathcal C$ construction through restoration from its $\mathcal H$ children, with reversal converting child returns into inputs. These routes leave all edge directions unchanged, and the direct additions in $\mathcal H$ calls are zero returns. Complete the candidate lists in each child graph first, then restore and verify every child candidate in the parent before taking a parent minimum. Simplify each restored walk and check all its required length bounds. For a positive return, also check the bounds involving $I_R$, $F_R$ and $W_R$, using its own actual arm on that path. After reversing every edge, the corresponding checks apply to inputs. The preceding proofs provide nonempty lists.

This convention preserves the comparison between a parent minimum and a child minimum plus its restoration cost. Positive restorations satisfy all their bounds by Lemma~\ref{lem:ambient-savings}. If an ordinary restoration exceeds a required parent bound, a proved candidate satisfying all required bounds is shorter; using it instead preserves the comparison.

\subsection{Selecting child paths}

The even-diameter argument in Section~\ref{sec:rebuilding} needs an exact relation between a parent path and its selected child path when the parent arm has length one. We therefore use one independent outer step: after completing and verifying the child candidate list, select its minimum and retain the prescribed restored path without taking another parent minimum. For a positive return restored through a length-one parent arm, this retains the exact increase of one proved below. The recursion already proved for \(\mathcal C\) is unchanged.

\begin{lemma}\label{lem:independent-root-step}
Use the subgraph $A$ formed by the red--blue edges, their endpoint tree routes to the root, and every length-one arm pair. Construct every child graph by the first rule, and use $\mathcal C$ on red components and the reversal of $\mathcal C$ on blue components. For each red external return, including one from a red vertex whose tree path begins with a complete arm, choose a minimum available $\mathcal C$ child return before restoring its actual attachment; use the reversed rule for a blue external input. If the root has one distinct neighbour, call $\mathcal C$ on the graph with the parent root deleted, rooted at that neighbour and with no designated arms, and append the parent edge to a minimum child return.

Under the two requirements on retained edges and their original endpoints stated above, suppose $R<\rho$ and $\tau=\rho+R-D\in[0,R]$. Both ordinary parent directions have length at most $A_R-g$, where $g=(2\tau-9)_+$. The original pointwise bounds remain valid, and a selected positive return with actual arm $w$ satisfies $I_R(z)$ and $\min\{W_R(w),F_R-g\}$ on the same path. A selected positive return restored from a child through a length-one parent arm has length exactly the chosen child minimum plus one. Every external return from another red vertex and every external input to a blue ordinary vertex has length at most its corresponding child minimum plus $2R$. Reversing the entire orientation exchanges inputs and returns.
\end{lemma}

\begin{proof}
The child calls are to the already proved construction $\mathcal C$; this is an independent outer step, not a change to its recursion. If deleting the root disconnects the graph, apply the construction separately to each component together with the root, as in Lemma~\ref{lem:rooted-orientations}. Assume $R\ge1$, since the case $R=0$ is empty. The first rule gives depths exactly one less, prescribed bound $R-1$, and parameter $\tau-1$. Lemma~\ref{lem:ambient-savings} bounds both child directions by $A_{R-1}-g$, using the original bound if the child parameter is negative.

Returns from the other red vertices and inputs to blue ordinary vertices restore their selected child returns with increase at most $2R$, proving both the ordinary bound and the child-minimum comparison. Inputs to the other red vertices, returns from blue ordinary vertices, and inputs to red ordinary vertices whose tree paths begin with complete arms restore in the tree direction and add one, giving $I_R(z)$. A return from a red ordinary vertex whose tree path begins with a complete arm adds one when its child return is zero at an actual parent-arm endpoint of depth one. This parent arm has length one by the degree-two argument above. If the child return is positive, its own $I$ bound increases by at most $2R-2$. In both cases the parent return retains $I_R(z)$. These bounds are at most $A_R-g$.

Suppose a positive return from a red ordinary vertex whose tree path begins with a complete arm uses a zero child return. Every depth-one attachment was identified with the child root. Consequently, the simple child path has not visited its final recorded attachment earlier, and its retained part does not contain the parent root. Appending the length-one parent arm is therefore simple and increases the length by exactly one. On this path the child depth bound gives $I_R(z)$, its general length bound gives $W_R(1)=A_{R-1}+1$, and its improved length bound gives $F_R-g=A_{R-1}-g+1$.
If both child and parent returns are positive, the actual child arm has length $w-1$, where $w\ge2$, and restoration again has net increase one. Thus the same route satisfies its pointwise $I$ bound and $W_R(w)\le W_R(2)=F_R-2R+3\le F_R-g$. A return ending with an ordinary root edge needs no positive-return bound.

For paths in $A$ use the original estimates if $g=0$. Otherwise $R\ge5$ and $g\le2R-9$, so $A_R-g\ge F_R+8\ge2R$ and $F_R-g\ge R$. In the one-distinct-neighbour case, the child bound $A_{R-1}-g$, its depth bound, and the appended edge give $F_R-g$, $I_R(z)$ and $W_R(1)$ together. Appending creates no repeated endpoint, so the exact increase of one is preserved. The neighbour itself has length one. A smaller actual eccentricity $R'<R$ is covered by $A_{R'}\le A_R-2R$ and $F_{R'}\le F_R-2R+2$ and the usual monotonicity.

Verified shortening or deletion of closed subwalks preserves the ordinary comparisons and the recorded attachment. A length-one positive restoration is already simple and satisfies every required bound: retain this prescribed route, with no further parent minimization or shortening. Its exact relation to the chosen child minimum is therefore unchanged. All attachment records are also retained when all edge directions are reversed.
\end{proof}

\section{Joining the rooted paths}\label{sec:joining-paths}

We now join an outside return, a path in the central subgraph, and an outside input. The goal is Lemma~\ref{lem:initial-estimate}: the initial orientation already has the required diameter unless one selected class of paths is too long. A nonnegative integer \(\delta\), defined below from those lengths, measures their saving from the general root bounds. A larger \(\delta\) makes the diameter estimate smaller.

Retain $d=2r+\varepsilon$ and $\rho=r+t+\varepsilon$, and write $L=\lceil d^2/2\rceil=2r(r+\varepsilon)+\varepsilon$. Assume $d\ge3$, and $t\ge1$ when $d$ is even; the excluded cases are settled in the final proof. The root bounds are $r-1+\varepsilon$ on $X$ and $r$ on $Y$. After restoration, a selected return runs from the source to its actual attachment, and a selected input runs from its actual attachment to the target.

\begin{lemma}\label{lem:tail-estimates}
Let a selected auxiliary root path have length $m$, actual attachment $s$, and length $l$ after restoration. With $\varphi=a$ on $X$ and $\varphi=d+\rho-b$ on $Y$, the restored path satisfies
\[
\begin{array}{c|cc}
&\text{source: }l-\varphi(s)&\text{target: }l+\varphi(s)\\
X&m-t+1-\varepsilon&m+t-1+2\varepsilon\\
Y&m-d-t+1-2\varepsilon&m+d+t+\varepsilon
\end{array}
\]
where each entry is an upper bound. The bounds also hold for an endpoint in $H$, with restored path length zero and numerical auxiliary length $m=w(s)$.
\end{lemma}

\begin{proof}
Restoration gives $l=m-w(s)$, including zero attachments. On $X$, the partition inequalities give $a+w\ge t+\varepsilon-1$ and $a-w\le t-1+2\varepsilon$. On $Y$, they give $w+\varphi\ge d+t-1+2\varepsilon$ and $\varphi-w\le d+t+\varepsilon$. Subtract or add these bounds to $m$. The same calculation with $m=w(s)$ proves the convention for endpoints in $H$, which adds no edge.
\end{proof}

For actual attachments $u,v$, join the restored source-to-$u$ path, a directed $u$--$v$ path in $H$, and the restored $v$-to-target path. A bound of $\varphi(v)-\varphi(u)+B$ for the path in $H$ cancels the potentials in Lemma~\ref{lem:tail-estimates}. Thus the attachment dependence cancels before we compare the remaining terms. Every segment belongs to the same orientation. If vertices repeat, including when both endpoints lie in one external component, delete closed subwalks.

We use the parameters $\gamma=(2t-11)_+$ and $\beta=(2t-9)_+$ for both parities.
For integer $t\ge0$,
\begin{equation}\label{eq:gamma-bounds}
t-\gamma\le5,\qquad 2t-\gamma\le11,\qquad 2t-2\gamma\le10.
\end{equation}
These follow from $\gamma=0$ for $t\le5$ and $\gamma=2t-11$ for $t\ge6$. Endpoints in $H$ obey the path bounds below. Their numerical lengths are at most the corresponding root bound $R$, and the required savings below are at most $(2R-7)_+$. Since $F_R-R=(R-1)^2\ge(2R-7)_+$, $A_R\ge F_R$, and $W_R(w)\ge R\ge w$, these lengths satisfy all the stated bounds.

In the definitions of $\delta$ below, all selected lengths are auxiliary root-path lengths, measured before restoration. For an endpoint $s$ in $H$, both directions use the numerical auxiliary length $m=w(s)$ and are classified as zero when $w(s)=0$ and positive when $w(s)>0$.

\noindent\textbf{Odd diameter.} Use $\mathcal H$ on $X$ and reversed $\mathcal C$ on $Y$, selecting inputs and returns as minima in their candidate families after adding all improvements proved above. Lemma~\ref{lem:ambient-savings} gives saving $g=(2t-7)_+\ge\beta\ge\gamma$ in both ordinary directions. Positive $X$ returns and positive $Y$ inputs also satisfy $F_r-g$ and their own actual-arm $W_r$ bounds. For the estimates below we may weaken the saving to $\gamma$. Only the zero-attachment pair $(Q_t,Q_{t+1})$ needs the exceptional path in $H$. Define $A_r-\delta$ as the largest selected auxiliary $Y$ input length actually using $Q_{t+1}$, including $Q_{t+1}$ itself with length zero. Then $\delta\ge g\ge\beta$.

\noindent\textbf{Even diameter.} Apply Lemma~\ref{lem:independent-root-step} on both sides and reverse $Y$. Its saving on $X$ is $(2t-11)_+=\gamma$, and on $Y$ it is $\beta$. Retain ordinary bounds $A_{r-1}-\gamma$ on $X$ and $A_r-\beta$ on $Y$; positive $X$ returns and positive $Y$ inputs additionally satisfy $F_{r-1}-\gamma$ and $F_r-\beta$, with their actual-arm bounds. Use the prescribed parent paths: take the child minimum before restoration and do not minimize again over arbitrary parent paths. This includes the one-distinct-neighbour case and preserves the exact length-one increase. For selected auxiliary $Y$ return lengths $Q_Y$ and auxiliary lengths of zero and positive inputs $P_{Y,0},P_{Y,+}$, define
\[
\delta=\min\{A_r-\max Q_Y,\ A_r-\max P_{Y,0},\ F_r-\max P_{Y,+}\}.
\]
An empty class contributes $A_r$. Every class saving, including the numerical lengths for endpoints in $H$, is an integer at least $\beta$. Thus $\delta\ge\beta$ in both parities.

We now estimate the diameter of the initial orientation. The next lemma also identifies the saving threshold beyond which rebuilding is unnecessary.

\begin{lemma}\label{lem:initial-estimate}
Suppose that the centre $p$, its farthest vertex $q$, and a shortest $p$--$q$ path $Q$ are chosen so that either the quantity $m$ corresponding to a noncritical crossing edge is at most $r+\varepsilon$, or $d$ is odd and some edge has covering radius at most $r$. In the edge-cover case, $G$ has a strong orientation of diameter at most $L+d-1$. Otherwise, form the initial central subgraph from the fixed shortest-path trees and the directed cycles through the reversed critical edge, and use the prescribed rooted orientations and selected paths in its external components. For odd $d$,
\begin{equation}\label{eq:initial-odd}
\diam(\rbjt{G})\le L+\max\{2r+2t+6-2\gamma,\ 3r+2t+1-\gamma-\delta\}.
\end{equation}
For even $d$, every $b_0\le\delta$ gives
\begin{equation}\label{eq:initial-even}
\begin{split}
\diam(\rbjt{G})\le L+\max\{&2t+3-2\gamma,\ 2r+2t+5-\gamma-b_0,\\
&4r+2t+2-2b_0,\ 3r+2t-\gamma-b_0-1\}.
\end{split}
\end{equation}
In either parity, the resulting orientation has diameter at most $L+d+18$ whenever $\delta\ge r+t-8$.
\end{lemma}

\begin{proof}
The edge-cover alternative is odd; Lemma~\ref{lem:edge-cover} gives diameter at most $2r^2+4r+1=L+d-1$. For odd $d$, Lemma~\ref{lem:tail-estimates} gives common terms $r^2-\rho-\gamma$ and $r^2+\rho-\gamma$ for the source and target paths, respectively. The additions for positive $X$ returns, zero $X$ returns, and $Y$ returns are $2,d,-1$; those for $X$ inputs, positive $Y$ inputs, and zero $Y$ inputs are $0,2,d$. Since $2r^2+d=L$, Lemmas~\ref{lem:central-core} and~\ref{lem:shallow-attachments} give excesses over $L$ of at most $2h+4-2\gamma$, $2h+3-2\gamma$, $2h+2-2\gamma$, and $2h-2\gamma$, according as neither end, only the source, only the target, or both ends use the indicated zero attachments; the double-zero critical pair is excluded. For noncritical double-zero pairs this uses the actual path length in $H$ $2h+1$ and the sum of the restored outside path lengths $2A_r-2\gamma=L-1-2\gamma$. All these cases are bounded by the first term of \eqref{eq:initial-odd}, since $h\le r+t+1$. The restored outside path lengths for the critical pair sum to at most $L-1-\gamma-\delta$, and Lemma~\ref{lem:axis-choice} bounds the connecting path length in $H$ by $2h+r$. This proves the second term. If $\delta\ge r+t-8$, \eqref{eq:gamma-bounds} bounds these two excesses by $2r+16$ and $2r+14$, both at most $d+18$.

For even $d$, the initial case has $T=2h$ and $K=0$. We retain a variable form of the path calculation for reuse after rebuilding in Section~\ref{sec:rebuilding}. This calculation uses only the selected root-path bounds and the assumed additional terms for paths in $H$. Let $b_0$, possibly negative, be a lower bound for the savings in the three classes of selected $Y$ paths in the orientation being estimated. Thus $b_0\le\delta$ for the present orientation; after rebuilding, the same calculation uses $b_0=\delta'$ for the new orientation. Suppose the general addition from $H$ is at most $d+T+K$, reduced to $T+K$ for a zero $X$ source and to $T+1$ for a zero $Y$ target. Here $T,K\ge0$. Lemma~\ref{lem:tail-estimates} bounds $l-\varphi(s)$ by $F_{r-1}-\gamma-t+1$, $A_{r-1}-\gamma-t+1$, and $A_r-b_0-2r-t+1$ for positive $X$, zero $X$, and ordinary $Y$, respectively. The bounds on $l+\varphi(s)$ for ordinary $X$, positive $Y$, and zero $Y$ are $A_{r-1}-\gamma+t-1$, $F_r-b_0+2r+t$, and $A_r-b_0+2r+t$, respectively. Adding the term from $H$ and subtracting $L$ gives the following table. Rows specify the source side and return type, and columns specify the target side and input type.
\begingroup\small\setlength{\arraycolsep}{3pt}
\begin{equation}\label{eq:even-path-costs}
\begin{array}{c|ccc}
&X&Y\text{ positive}&Y\text{ zero}\\
X\text{ positive}&-2r+3+T+K-2\gamma&5+T+K-\gamma-b_0&5+T-\gamma-b_0\\
X\text{ zero}&-2r+T+K-2\gamma&2+T+K-\gamma-b_0&2r+T+2-\gamma-b_0\\
Y&T+K-b_0-\gamma&2r+2+T+K-2b_0&2r+2+T-2b_0
\end{array}
\end{equation}
\endgroup
For example, the top-left entry is $F_{r-1}+A_{r-1}+d+T+K-2\gamma-L$. In the first two columns, replacing a positive $X$ return by a zero return adds $2r-3$ to the restored path length and subtracts $2r$ from the term contributed by $H$, reducing the entry by three. For a positive $X$ or ordinary $Y$ return, replacing a positive $Y$ input by a zero input reduces the total by $K$. When both paths have zero attachments, the entry uses $T+1$ directly for the term contributed by $H$.

For the initial orientation, use $T=2h$ and $K=0$ in \eqref{eq:even-path-costs}. The double-zero entry improves to $T+r-1-\gamma-b_0$: the restored outside path lengths sum to at most $A_{r-1}+A_r-\gamma-b_0=L-\gamma-b_0$, and Lemmas~\ref{lem:shallow-attachments} and~\ref{lem:axis-choice} bound both relevant actual connecting path lengths in $H$ by $T+r-1$, as $r\ge2$. The resulting maximum is
\[
\begin{split}
\max\{-2r+3+2h-2\gamma,\ 5+2h-\gamma-b_0,\ 2r+2+2h-2b_0,\ 2h+r-1-\gamma-b_0\}.
\end{split}
\]
Now $2h\le2r+2t$ proves \eqref{eq:initial-even}. If $\delta\ge r+t-8$, take $b_0=r+t-8$, which is allowed even when negative. By \eqref{eq:gamma-bounds}, the four terms are at most $13$, $r+18$, $2r+18$, and $2r+12$, respectively, and each is at most $d+18$.
\end{proof}

\section{Rebuilding the central subgraph}\label{sec:rebuilding}

Only the case of insufficient saving remains. Lemma~\ref{lem:farthest-extraction} uses a long selected path to find a new farthest vertex $w$ of the same centre $p$, together with a connected vertex set $U$ and an attachment $\xi$. Its lower bound concerns paths whose internal vertices lie in $U$. Lemma~\ref{lem:rebuilt-core} uses this property to rebuild $H$ along a shortest $p$--$w$ path through $\xi$. The directed cycles are formed as in Section~\ref{sec:central-core}, but the retained paths on the $Y$ side may be longer than shortest paths. Lemma~\ref{lem:rebuilt-distances} bounds the connecting paths in this fixed orientation.

The old saving $\delta$ controls the bounds for the rebuilt subgraph $H$, whereas the new saving $\delta'$ controls the lengths of its outside paths. Lemma~\ref{lem:rebuilt-estimate} combines these bounds and proves the required diameter whenever $\delta'\ge\delta$. This is why we choose a minimum saving among all choices obtainable by finitely many permitted rebuildings: the next rebuilding remains among those choices and therefore cannot have a smaller saving. The precise choices and path-selection rules are specified in the final proof.

By Lemma~\ref{lem:initial-estimate}, the only remaining range for the initial saving is
\begin{equation}\label{eq:deficit-range}
\beta\le\delta\le r+t-9\le2r-9.
\end{equation}
The same range will be used in every subsequent extraction. Define $\ell=r-2+\varepsilon$ and $K=s=\lfloor\delta/2\rfloor+2$. Then $2\le s\le r-3\le\ell-1<\ell$, $\delta\le2s-3$, and $\delta\le2\ell-5$.
An empty range requires no extraction; a nonempty range has $r+t\ge9$, hence $r\ge5$.

\begin{lemma}\label{lem:farthest-extraction}
Fix a centre $p$, a farthest vertex, a shortest path $Q$ between them, and a central subgraph $H$ containing every crossing edge. Orient $H$ and the auxiliary graphs as prescribed above, add all permitted improvements to the candidate path families, and select the paths and their actual attachments as in Section~\ref{sec:joining-paths}. Let $\delta$ be their saving. Suppose \eqref{eq:deficit-range} holds. A vertex attaining the selected maximum in a path class whose saving is $\delta$ lies in a connected set $U$ of real vertices with an attachment $\xi$ such that $a(\xi)=t+2$. Every vertex $v$ of $U$ and every attachment of $U$ has a real tree path from $\xi$ of length $a(v)-t-2\le\ell$. There is one vertex $w\in U$ with $a(w)=\rho$ such that every path from $w$ to an attachment other than $\xi$, with internal vertices in $U$, has length at least $2\ell-K$. A shortest $\xi$--$w$ path has length $\ell$ and all internal vertices in $U$. In particular $0\le K<\ell$.
\end{lemma}

\begin{proof}
The proof first brings the long selected path, in either parity, to a \(\mathcal C\) call of actual eccentricity \(\ell+1\). Its selected vertex \(y\) will lie in an external component \(U\) of the other red vertices and have minimum return at least \(A_{\ell+1}-\delta\). We then descend to choose a single vertex \(w\), lift distance bounds for all positive attachments to this same vertex, and verify that it is a farthest vertex of \(p\).

For odd $d$, choose $y$ attaining $A_r-\delta$ among the selected $Y$ inputs through $Q_{t+1}$, and reverse the auxiliary orientation. It becomes a minimum return in $\mathcal C$. By \eqref{eq:deficit-range}, its length exceeds $F_r$, $2r$, and $A_{r-1}$: the differences from $F_r$ and $A_{r-1}$ are at least eight and nine, while $A_r-\delta-2r\ge r^2-3r+9>0$. These exclude the smaller-bound classes, vertices of the subgraph $A$ in that recursive call, and smaller actual eccentricity. Thus $y$ lies in an external component of the other red vertices of a $\mathcal C$ call of eccentricity $r=\ell+1$.

For even $d$, a class attaining $\delta$ is nonempty, since an empty class contributes $A_r>2r-9$. For an ordinary $Y$ return or zero $Y$ input, the selected length $A_r-\delta$ again exceeds $F_r,2r,A_{r-1}$. Before reversal it is, respectively, an input to a blue ordinary vertex or a return from one of the other red vertices. Lemma~\ref{lem:independent-root-step} restores it from a minimum $\mathcal C$ child return with increase at most $2r$, so the child minimum is at least $A_{r-1}-\delta$. For a positive $Y$ input, the length is $F_r-\delta>r$, since $F_r-\delta-r\ge(r-2)^2+6$, excluding prescribed arm endpoints, which lie in the subgraph $A$ of that recursive call. Its actual arm has length one, since length at least two would give $W_r(2)=F_r-2r+3<F_r-\delta$. The prescribed restoration, including the one-distinct-neighbour case, is exactly the child minimum plus one. Thus that minimum equals $F_r-\delta-1=A_{r-1}-\delta$. Smaller actual parent eccentricity is excluded by $F_{r-1}<F_r-\delta$, since $F_r-F_{r-1}=2r-2>\delta$. In all three cases the child has eccentricity at most $r-1$ and ordinary depths exactly one less than the parent. The lower bound exceeds $F_{r-1}$ and $A_{r-2}$, whose differences from $A_{r-1}$ are $2r-3$ and $2r-2$, both greater than $\delta$. It also exceeds $2(r-1)$, because $A_{r-1}-(2r-9)-2(r-1)=(r-2)(r-3)+5>0$. Hence its eccentricity is $r-1=\ell+1$, and $y$ is in an external component $U$ of the other red vertices.

By Lemmas~\ref{lem:auxiliary-attachments} and~\ref{lem:rooted-orientations}, the ordinary coordinate in this $\mathcal C$ call is $a(v)-t-1$. In the one-distinct-neighbour case this follows by deleting the parent root, whose neighbour has coordinate $a=t+1$. Let $\xi$ be the first vertex of the branch containing $U$. Its depth is one, so $a(\xi)=t+2$. All vertices of $U$ and all its attachments have this red label not belonging to an arm; their branch paths from $\xi$ are real and have length $a(v)-t-2\le\ell$.

Identify attachments of parent depth one with the child root, add two arms of length $j-1$ at each attachment of parent depth $j\ge2$, and apply $\mathcal H$. The exact ordinary coordinate is now $a(v)-t-2$, and eccentricity is at most $\ell$. If $Q(y)$ is the minimum candidate return length, restoration gives a parent candidate of length at most $Q(y)+2(\ell+1)$, so $Q(y)\ge A_{\ell+1}-\delta-2(\ell+1)=A_\ell-\delta$.
Since $\delta<2\ell$, smaller eccentricity would contradict $A_{\ell-1}=A_\ell-2\ell$; the actual eccentricity is $\ell$. Also $Q(y)>F_\ell$, since $A_\ell-\delta-F_\ell=2\ell-1-\delta\ge4$, so no root geodesic to $y$ begins with a complete arm: its tree path would then begin with such an arm and give a return of length at most $I_\ell(z)\le F_\ell$. Every root geodesic therefore starts with an ordinary root edge, which in this child can only record $\xi$. Thus $\xi$ is an actual attachment and the unique zero attachment of $U$; all others have positive arms.

\noindent\textit{Descent and the choice of $w$.} Keep the selected vertex $y$ fixed throughout the descent; $Q(y)$ denotes its minimum candidate return length in the current call. Maintain actual eccentricity $R$ and minimum return at least $A_R-\delta$ at $y$. When $R\ge s$, $\delta\le2R-3$, so the lower bound exceeds $F_R,A_{R-1},2R$; the last follows from $A_R-\delta-2R\ge(R-1)(R-2)+1>0$. Thus $y$ lies in an external component of the other red vertices. If the root separates the graph, restrict to the component containing $y$; the same lower bound forces its eccentricity to remain $R$. The one-distinct-neighbour case has an $I_R$ bound and is excluded. The child with added arms of length one has eccentricity at most $R-1$. If it were at most $R-2$, its minimum would restore within $A_{R-2}+2R=A_R-2R+2<A_R-\delta$, a contradiction. Hence its actual eccentricity is $R-1$, and its minimum return at $y$ is at least $A_{R-1}-\delta$.

At the call of eccentricity $s$, take its child with added arms of length one, of eccentricity $s-1$, and choose any vertex $w$ of depth $s-1$. Keep this $w$ fixed in every ancestor and in the distance induction below. New arms have length one and no internal vertices; earlier artificial vertices lie in the subgraph $A$ of that recursive call and cannot lie in a component of the other red vertices. Hence $w$ is real or a real attachment. The child-depth inequality makes its parent depth at least $s$; the parent eccentricity makes it at most $s$, so equality holds. Repeating this argument upwards shows that $w$ has depth $R$ in every ancestor of eccentricity $R$.

\noindent\textit{A common distance bound for all positive attachments.} For this same $w$, we prove the following statement by induction up the recursive calls: every path from $w$ to a positive attachment in a call of root eccentricity $R$ has length at least $2R-K$, provided it uses only the real edges retained in that call and avoids its auxiliary root and artificial arm interiors.

For the base call $R=s$, let $V$ be the last external component of the other red vertices and $\tau$ its first branch vertex. Fix a positive parent attachment $e$ and a $w$--$e$ path using only the retained real edges and avoiding the auxiliary root and artificial arm interiors. If $w\in V$, read its first exit attachment $z$; if $w$ is itself an attachment of $V$, take $z=w$. For an exit $z$ of parent depth at least two, its child arm has length one, so the initial segment has length at least $s-2$. Its red label does not belong to an arm and is distinct from and nonadjacent to the label of the main arm ending at $e$, so the remainder has length at least two. The total is at least $s=2s-K$. The exit $\tau$ of parent depth one gives an initial segment of length at least $s-1$ and the same two-edge remainder, hence a larger total. If $w=z$ is an attachment, its child depth is at most one, so $s=2$; the zero initial segment still has length $s-2$.

Now pass from a child of eccentricity $M-1$ to its parent of eccentricity $M$, with a component $V$ of the other red vertices and first branch vertex $\tau$ of depth one. Since $M>s$, the vertex $w$ lies in $V$: its child depth is $M-1\ge s\ge2$, while every child attachment has depth at most one. Consider a path from $w$ to a positive parent attachment $e$ using only the real edges retained in the parent and avoiding its auxiliary root and artificial arm interiors. If its first exit from $V$ has parent depth at least two, this exit is a positive attachment in the child. The initial segment uses retained real child edges and avoids the child root and artificial arm interiors, so the induction hypothesis bounds it by $2(M-1)-K$; label nonadjacency adds at least two, giving $2M-K$. If its first exit is $\tau$, this exit is identified with the child root, so we instead use the long minimum return at the fixed vertex $y$. The depth $M$ of $w$ gives an initial segment of length at least $M-1$. Let the remaining length be $l\ge2$, and let $z$ be the parent depth of $y$. Join the remaining path to the real $\tau$--$y$ tree path of length $z-1$, and delete closed portions. This gives an $e$--$y$ path satisfying the hypotheses of Lemma~\ref{lem:positive-attachment-path}, of length at most $l+z-1$. Lemma~\ref{lem:positive-attachment-path} yields $A_M-\delta\le Q(y)\le I_M(z)+2l\le F_M+2l$. Consequently $l\ge M-\lfloor(\delta+1)/2\rfloor$, and the entire path has length at least $2M-\lfloor(\delta+3)/2\rfloor\ge2M-(\lfloor\delta/2\rfloor+2)=2M-K$. After a restriction at a separating root, there is no path from $w$ to a positive attachment in another root component using only the retained real edges and avoiding the root and artificial arm interiors. This completes the distance induction for the same vertex $w$.

\noindent\textit{Real vertices and the internal geodesic.} In the first auxiliary graph formed from \(U\), $\xi$ is identified with the root and every other original attachment is an endpoint of a positive arm whose tree branch has the label of its main arm. An external component of the other red vertices is adjacent neither to a branch whose tree path begins with an arm nor to the root, whose neighbours lie in the subgraph $A$ of that recursive call. Its vertices and all its real attachments therefore lie strictly inside $U$. Later calls introduce no real vertices and retain this property; the label-tree paths after their first root edge use real edges. Because $s<\ell$, the final call occurs after this first descent. Thus the selected $w$, including the possible attachment at $s=2$, belongs to $U$.

Its depth $\ell$ in the first auxiliary graph formed from \(U\) gives $a(w)=t+2+\ell=\rho$. Every path from $w$ to an original attachment $e\ne\xi$ whose internal vertices lie in $U$ uses only the real edges retained in that first auxiliary graph and avoids its root and artificial arm interiors. Such a path cannot pass through $\xi$, since its internal vertices lie in $U$. Its length is therefore at least $2\ell-K>\ell$. The real $\xi$--$w$ path in the branch tree of the $\mathcal C$ call of eccentricity $\ell+1$ has length $\ell=a(w)-a(\xi)$, so it is a geodesic in $G$. If it had an internal vertex outside $U$, its last entry into $U$ would pass through an attachment $e\ne\xi$ and leave an $e$--$w$ path with all internal vertices in $U$, of length at most $\ell$, a contradiction.
\end{proof}

We now use the vertex $w$, the set $U$, and its attachment $\xi$ supplied by Lemma~\ref{lem:farthest-extraction}. The new central subgraph is again a union of cycles through the reversed critical edge. The change is in the tree on the $Y$ side: its paths first approach $\xi$ and then join the new shortest path $Q$. We control both their excess over the true distance to $w$ and the vertices of $Q$ at which they first arrive.

\begin{lemma}\label{lem:rebuilt-core}
Let $p$ be a centre, and let $U$ be a connected vertex set with attachment $\xi$. Suppose $a(\xi)=P=t+2$, $w\in U$ satisfies $a(w)=\rho$, and $\ell=\dist(\xi,w)=r-2+\varepsilon$. Assume that a $\xi$--$w$ geodesic has all its internal vertices in $U$, and that every vertex $v$ of $U$ or its attachment set has a $\xi$--$v$ path of length $a(v)-P\le\ell$. Suppose also that every path from $w$ to an attachment other than $\xi$, with internal vertices in $U$, has length at least $2\ell-K$, where $K\ge0$.

Choose the new shortest path $Q$ by concatenating a $p$--$\xi$ geodesic and the prescribed $\xi$--$w$ geodesic. Redefine $b(v)=\dist(w,v)$ and $X,Y$ using the same threshold $\theta$. There is a connected bridgeless subgraph $H$ containing all crossing edges, with a strong orientation obtained by directing each noncritical crossing-edge path forwards and its intervening segment of $Q$ backwards. Its retained $X$ paths lie in one shortest-path tree rooted at $p$. Its retained $Y$ paths lie in one tree directed towards $w$ before the edges of $Q$ are reversed. For a retained $Y$ vertex $v$, let $D(v)$ be its distance to $w$ in this tree, let $\eta(v)=D(v)-b(v)$, and let $Q_{\nu(v)}$ be the first vertex of $Q$ on its tree path. Then
\begin{equation}\label{eq:rebuilt-tree-bounds}
0\le\eta(v)\le K,\qquad 2\nu(v)+\eta(v)\le T,\qquad T=2P+K.
\end{equation}
\end{lemma}

Here $\eta(v)$ is the excess of the tree-path length over the true distance to $w$, while $\nu(v)$ determines the length of a backwards segment of $Q$. The combined bound in \eqref{eq:rebuilt-tree-bounds} controls the extra tree length together with the backwards segment of $Q$.

\begin{proof}
The chosen concatenation has length $P+\ell=\rho=a(w)$, so it is a $p$--$w$ geodesic, with $\xi=Q_P$. As in Section~\ref{sec:central-core}, its unique crossing edge is $Q_cQ_{c+1}$. We first construct the new tree on $Y$, verify \eqref{eq:rebuilt-tree-bounds} for the paths that will be retained, and then form $H$ from these paths and a shortest-path tree on $X$.

Define $B(v)=\ell+\dist(\xi,v)$. The triangle inequality gives $B(v)\ge b(v)$. For $v\notin U$, a $w$--$v$ geodesic first leaves $U$ at an attachment $z$. If $z=\xi$, it has a $w$--$\xi$ prefix of length $\ell$, so $B(v)\le b(v)$. Otherwise the prefix has length at least $2\ell-K$. By hypothesis, $\dist(\xi,z)\le\ell$, so replacing this prefix by a route through $\xi$ adds at most $K$ edges. Thus $0\le B(v)-b(v)\le K$ for every $v\notin U$.

For $v\in U$ or an attachment of $U$, the prescribed $\xi$--$v$ path gives $B(v)\le\ell+a(v)-P$. Therefore $a(v)-b(v)\ge a(v)-B(v)\ge P-\ell=\theta+6-3\varepsilon$. Since $a-b$ changes by at most two across an edge, each such vertex and all its neighbours lie in $Y$. In particular, every crossing endpoint is outside $U$.

Let $Z=\{v:a(v)-B(v)>\theta\}$. Since $B\ge b$, we have $Z\subseteq Y$. The preceding inequality shows that $U$ and all its attachments lie in $Z$. The $Y$ segment of $Q$ also lies in $Z$: at $Q_j$, the value of $a-B$ is $2j-\rho$ for $j\le P$, and $P-\ell$ for $j\ge P$. Choose a shortest-path tree rooted at $\xi$ containing both segments of $Q$ from $\xi$ to its ends, and a shortest-path tree rooted at $w$ containing $Q$.

The next vertex on the route from each vertex of $Y\setminus\{w\}$ is prescribed as follows. A vertex of $Y\setminus Z$ follows its $w$-tree edge until the route first enters $Z$. Along these edges, $b$ decreases by one and $a$ decreases by at most one, so $a-b$ does not decrease and the route stays in $Y$. A vertex of $Z$ outside $Q$ follows its $\xi$-tree edge until the route first meets $Q$. Along these edges, $B$ decreases by one and $a$ decreases by at most one, so $a-B$ does not decrease and the route stays in $Z$. Once on $Q$, the route follows $Q$ towards $w$. Every route reaches $w$: before entering $Z$, its distance to $w$ decreases; before meeting $Q$, its distance to $\xi$ decreases; afterwards its index on $Q$ increases. Because the next vertex depends only on the current vertex, routes have the same continuation whenever they meet. Consequently the edges to these next vertices form one tree on $Y$, directed towards $w$.

Fix a crossing endpoint $y\in Y$. Let $z$ be the first vertex of $Z$ on its route, reached after $q_0$ edges. The initial segment belongs to the $w$-tree, so $b(y)=q_0+b(z)$. Moreover, $z\notin U$: the route starts outside $U$, and any first entry into $U$ is preceded by an attachment, which is already in $Z$. Suppose the $\xi$-tree path from $z$ first meets $Q$ at $Q_{P+g}$. Since this tree contains $Q$, its segment to $Q_{P+g}$ has length $\dist(\xi,z)-|g|$, and the remaining segment of $Q$ has length $\ell-g$. It follows that $D(y)=q_0+B(z)-|g|-g$, and hence $\eta(y)+2\max\{g,0\}=B(z)-b(z)\le K$. The tree path is a path in $G$, so $\eta(y)\ge0$. Since $\nu(y)=P+g$, these inequalities prove \eqref{eq:rebuilt-tree-bounds} at $y$.

Before the first vertex of $Q$, the index $\nu$ of that vertex is constant. If $v'$ is the next vertex after $v$, we have $D(v')=D(v)-1$ and $b(v')\ge b(v)-1$, so $\eta(v')\le\eta(v)$. Therefore \eqref{eq:rebuilt-tree-bounds} holds at every vertex outside $Q$ on the tree suffix of a crossing endpoint.

We now form $H$. Choose a shortest-path tree rooted at $p$ containing $Q$. For each noncritical crossing edge $xy$, with $x\in X$ and $y\in Y$, let $Q_\lambda$ be the last vertex of $Q$ on the $p$--$x$ tree path, and let $Q_\nu$ be the first vertex of $Q$ on the new $y$--$w$ tree path. The prefix to $x$ stays in $X$, whereas the suffix from $y$ stays in $Y$. Thus $\lambda\le c<\nu$, and the path from $Q_\lambda$ through $xy$ to $Q_\nu$ has all its internal vertices outside $Q$. Direct this path forwards and the intervening segment of $Q$ backwards. This produces a directed cycle through $Q_{c+1}\to Q_c$.

These cycle directions agree: outside $Q$ they follow the same two trees, every noncritical crossing edge is directed from $X$ to $Y$, and the edges of $Q$ are directed backwards. The construction needs the $Y$ paths to stay in $Y$ and to have a common continuation when they meet; it does not require them to be geodesics. There is at least one noncritical crossing edge, since otherwise the critical edge would be a bridge of $G$. Let $H$ be the union of all the resulting cycles. They share the reversed critical edge, so this union is strongly connected; every edge lies on a cycle, so its underlying graph is bridgeless. It contains all noncritical crossing edges and also the critical edge.

Finally, $Q\cap H$ runs from the smallest index $\lambda$ to the largest index $\nu$ arising from a crossing edge. At each vertex of $Q\cap H\cap Y$, its tree path follows $Q$ to $w$, so $\eta=0$ and $\nu$ is its own index. Every such index is no greater than $\nu$ for some crossing endpoint. Applying \eqref{eq:rebuilt-tree-bounds} at that endpoint proves \eqref{eq:rebuilt-tree-bounds} on $Q\cap H$ as well.
\end{proof}

The new $H$ and its orientation are now fixed. We next estimate directed paths using the two retained trees and $Q$ with its edges reversed, rather than choosing new paths separately for different vertex pairs. Continue to use the potential $\varphi$, with $b(v)=\dist(w,v)$. The conditions $b(u)=r+1$ and $a(v)=c+1$ therefore use the true distances in $G$, not the possibly larger distance $D(v)$ in the retained $Y$ tree.

\begin{lemma}\label{lem:rebuilt-distances}
For the fixed strong orientation of the rebuilt central subgraph $H$, every ordered pair \(u,v\in H\) satisfies $\dist_{\rbjt{H}}(u,v)\le\varphi(v)-\varphi(u)+d+T+K$. The additional term can be replaced by \(T+K+\varepsilon\) when $u\in X$ and $b(u)=r+1$, and by \(T+1-\varepsilon\) when $v\in Y$ and $a(v)=c+1$. When \(d\) is odd and both conditions hold, their actual directed distance is at most \(T+1\), except that $\dist_{\rbjt{H}}(Q_t,Q_{t+1})\le d+T-1$.
\end{lemma}

\begin{proof}
We first join entry and exit segments through $Q$ with its edges reversed. We then check sources or targets on $Q$, before treating the conditions $b(u)=r+1$ and $a(v)=c+1$ and the exceptional pair. All of these routes use the fixed trees from Lemma~\ref{lem:rebuilt-core}.

For a noncritical edge \(f=x_fy_f\), let \(m_f\le d\) be the quantity defined for a crossing edge in Lemma~\ref{lem:central-core}. Let $Q_{\lambda_f}$ be the last vertex of $Q$ on its $p$--$x_f$ tree path and $Q_{\nu_f}$ the first vertex of $Q$ on its $y_f$--$w$ tree path. Its forward path, whose internal vertices lie outside $Q$, has length \(m_f+\eta(y_f)+\nu_f-\lambda_f\). First suppose the source is not in $Q\cap X$ and the target is not in $Q\cap Y$. Choose \(f\) with its prefix through \(u\) when \(u\in X\), and \(g\) with its suffix through \(v\) when \(v\in Y\). For \(v\in X\), let $Q_{\lambda_v}$ be the last vertex of $Q$ on its $p$--$v$ tree path, taking \(\lambda_v=a(v)\) when $v\in Q$. The entry and exit segments have the following lengths:
\[
\begin{array}{r@{\;:\;}l@{\qquad}l}
u\in X&u\to Q_{\nu_f}&-\varphi(u)+m_f+\eta(y_f)+\nu_f\\
u\in Y&u\to Q_{\nu(u)}&-\varphi(u)+d+\eta(u)+\nu(u)\\
v\in X&Q_{\lambda_v}\to v&\varphi(v)-\lambda_v\\
v\in Y&Q_{\lambda_g}\to v&\varphi(v)+m_g-d+\eta(y_g)-\eta(v)-\lambda_g.
\end{array}
\]
These identities follow from the tree lengths, \(D=b+\eta\), and the definition of $m_f$; the second also holds for a source on $Q$, with an empty segment. By \eqref{eq:rebuilt-tree-bounds}, both entry lengths are at most \(-\varphi(u)+d+T-\nu\), and both exit lengths are at most \(\varphi(v)+K-\lambda\). Here \(\nu\ge c+1>\lambda\); inserting the backwards segment of $Q$ of length \(\nu-\lambda\) gives the claimed bound, with the extra saving \(2\lambda\).

We check the endpoints on $Q$ explicitly. For a source \(Q_i\in X\) and a target outside $Q\cap Y$, if \(\lambda\le i\), go directly backwards to \(Q_\lambda\). The additional term is at most \(2i+K-2\lambda\le T+K\). If \(\lambda>i\), choose \(f\) with \(\lambda_f\le i\), which exists by the left endpoint of $Q\cap H$. Use its cycle route before leaving at \(Q_\lambda\). Compared with the preceding calculation, the extra term is \(2i-2\lambda_f-2\lambda\le0\), so the general bound still holds.

Because $Q$ is directed towards smaller indices, reaching it below the target requires another forward crossing route. Now let the target be \(Q_j\in Y\), initially with the source outside $Q\cap X$. If the first vertex $Q_\nu$ of $Q$ reached by the entry path satisfies \(\nu\ge j\), the entry segment followed by $Q$ backwards has additional term at most \(T-2j\). If \(\nu<j\), choose \(g\) with \(\nu_g\ge j\), follow $Q$ backwards to reach \(Q_{\lambda_g}\), then its crossing route and $Q$ backwards to \(Q_j\). The coarser entry bound \(-\varphi(u)+d+K+\nu\) gives an additional term at most $d+K+2\nu-2j+(2\nu_g+\eta(y_g))-2\lambda_g\le d+K+T-2$.

If the source is also on $Q$, say \(Q_i\in X\), choose \(g\) with \(\nu_g\ge j\). When \(\lambda_g\le i\), the route through \(g\) has additional term $m_g-d+(2\nu_g+\eta(y_g))+2i-2\lambda_g-2j\le T$. Otherwise choose \(f\) with \(\lambda_f\le i\). If \(\nu_f\ge j\), use \(f\) instead. If \(\nu_f<j\), use both crossing routes; the additional term is
\[
m_f-d+m_g+\eta(y_f)+(2\nu_g+\eta(y_g))+(2i-2\lambda_f-2\lambda_g)+(2\nu_f-2j)\le d+K+T.
\]
The last two parentheses are nonpositive. This proves the general bound for all four combinations of endpoints on or outside $Q$.

For $u\in X\setminus V(Q)$ with $b(u)=r+1$, the first edge of a shortest $u$--$w$ path in $G$ crosses to \(Y\) and has \(m_f=a(u)-c\le\varepsilon\). All crossing edges are retained, so use this edge with its retained tree suffix. In both entry estimates, \(d\) is replaced by \(\varepsilon\); the above routes, including those ending in $Q\cap Y$, give \(T+K+\varepsilon\). The only vertex $u\in Q\cap X$ with $b(u)=r+1$ is \(Q_c\). For every crossing edge, $\lambda\le c$, so its direct routes along $Q$ give at most \(T+K\), including a target on $Q$.

For $v\in Y\setminus V(Q)$ with $a(v)=c+1$, the last edge of its selected \(p\)-tree path is a crossing edge \(g\), with \(y_g=v\) and \(m_g\le1-\varepsilon\). Thus the exit length is at most \(\varphi(v)+1-\varepsilon-d-\lambda_g\), with no error difference. It cancels the entry term \(d\), giving \(T+1-\varepsilon\). For a source in $Q\cap X$, use \(2i\le T\) in the direct case, and the same nonpositive correction \(2i-2\lambda_f-2\lambda_g\) otherwise. If the target is \(Q_{c+1}\), every index $\nu$ arising from a crossing edge is at least \(c+1\), so only the direct route to the target on $Q$ is needed; its additional term is at most \(T\).

Finally assume \(\varepsilon=1\), $u\in X$ with $b(u)=r+1$, and $v\in Y$ with $a(v)=c+1$. If neither is on $Q$, the quantities associated with their entry and exit edges are \(a(u)-t\) and zero, and the displayed segment identities give actual length $1+\eta(y_f)+2\nu_f-2\lambda_g\le T+1$.
If \(u=Q_t\) but \(v\) is outside $Q$, its selected \(p\)-tree predecessor has depth \(t\), and the direct route has length \(2(t-\lambda_g)+1\le T+1\). If \(v=Q_{t+1}\) but \(u\) is outside $Q$, the entry route followed backwards to \(v\) has length \(\eta(y_f)+2\nu_f-2t-1\le T+1\). For the remaining pair, take any noncritical \(f\), follow $Q$ backwards from \(Q_t\) to \(Q_{\lambda_f}\), use its crossing route, and return to \(Q_{t+1}\). The length is $2(\nu_f-\lambda_f)+m_f+\eta(y_f)-1\le d+T-1$. All routes use the same fixed tree orientation. Deleting closed subwalks can only shorten them.
\end{proof}

We now combine the rebuilt bounds within $H$ with the rooted paths already proved. The numerical condition in the next lemma is the only one needed for the final argument using a choice with minimum saving.

\begin{lemma}\label{lem:rebuilt-estimate}
Let the prescribed orientation and selected candidate paths have saving $\delta$ satisfying \eqref{eq:deficit-range}, and define $P=t+2$, $K=\lfloor\delta/2\rfloor+2$, and $T=2P+K$. Let $U$ be a connected set of real vertices with an attachment $\xi$ such that $a(\xi)=P$, and let $w\in U$ satisfy $a(w)=\rho$. Suppose every vertex $v$ of $U$ or its attachment set has a real tree path from $\xi$ of length $a(v)-P\le\ell$, and a $\xi$--$w$ geodesic has all its internal vertices in $U$. Suppose also that every path from $w$ to an attachment other than $\xi$, with internal vertices in $U$, has length at least $2\ell-K$. Keep the centre $p$ fixed, and choose the new shortest path $Q$ by concatenating a $p$--$\xi$ geodesic and this $\xi$--$w$ geodesic. Form the rebuilt central subgraph with its fixed trees and directed-cycle orientation, and use the same prescribed rooted orientations and path choices in its external components. If the new saving satisfies $\delta'\ge\delta$, the resulting orientation has diameter at most $L+d+18$.
\end{lemma}

\begin{proof}
Lemma~\ref{lem:rebuilt-core} applies with $P=t+2$, $K=\lfloor\delta/2\rfloor+2$, and $T=2t+4+K$. All crossing edges are retained, so Lemma~\ref{lem:auxiliary-attachments} gives the same exact auxiliary coordinates and root bounds. Since the centre and radius are unchanged, Lemma~\ref{lem:ambient-savings} still applies for odd $d$, and Lemma~\ref{lem:independent-root-step} for even $d$. The same path-selection rules from Section~\ref{sec:joining-paths} therefore give $\delta'\ge\beta$. Use Lemma~\ref{lem:rebuilt-distances} for paths in $H$ and Lemma~\ref{lem:tail-estimates} for the restored outside paths; the latter uses true distances, even when the retained $Y$ tree is not geodesic. All paths in this comparison belong to the new orientation. We use $2K\le\delta+4$, $K\le r-3$, and $K-\delta=2-\lceil\delta/2\rceil$. Since $\delta\ge\beta$, also $\lceil\delta/2\rceil\ge(t-4)_+$: this is immediate for $t\le4$, while for integer $t\ge5$ it follows from $\delta\ge2t-9$.

For odd $d$, the ordinary additional term for paths in $H$ is $d+2P+2K$, the zero-$X$ addition is $2P+2K+1$, and the zero-$Y$ addition is $2P+K$. The noncritical double-zero actual path length in $H$ is at most $2P+K+1$. Thus the four noncritical calculations of Lemma~\ref{lem:initial-estimate}, with $2h$ replaced by $2P+2K$, give excess over $L$ at most $2t+8+2K-2\gamma\le\delta+12+(2t-2\gamma)\le2r+13$. For the critical pair the restored outside paths total at most $L-1-\gamma-\delta'$, and the path length in $H$ is at most $d+2P+K-1$. Its excess is at most
\[
\begin{aligned}
2r+2t+3+K-\gamma-\delta'&\le d+2t+4-\gamma-\lceil\delta/2\rceil\\
&\le d+2t+4-\gamma-(t-4)_+\le d+13.
\end{aligned}
\]
For the last inequality, $t\le4$ gives $2t+4\le12$, whereas $t\ge5$ gives $t+8-\gamma\le13$ by \eqref{eq:gamma-bounds}.

For even $d$, use the calculation \eqref{eq:even-path-costs} with $b_0=\delta'$, the saving in the new orientation. Replacing $\delta'$ by $\delta$ only increases its entries. The positive-$X$/positive-$Y$ entry exceeds the top-left entry by $2r+2+\gamma-\delta\ge11+\gamma>0$, and exceeds the bottom-left entry by five. In the first two columns a zero-$X$ return reduces the corresponding positive-$X$ entry by three. The top-right entry is smaller than the top-middle entry by $K$, whereas the double-zero entry exceeds the top-middle entry by $2r-3-K\ge r>0$. The bottom-right entry is at most the bottom-middle entry. Thus all nine entries are bounded by the maximum of $d+T+2-\gamma-\delta$ and $d+2+T+K-2\delta$. After substitution, these are respectively
\[
\begin{aligned}
d+2t+8-\gamma-\lceil\delta/2\rceil&\le d+2t+8-\gamma-(t-4)_+\le d+17,\\
d+2t+10-2\lceil\delta/2\rceil&\le d+2t+10-2(t-4)_+\le d+18.
\end{aligned}
\]
Indeed, for $t\le4$ the two excesses over $d$ are at most $16$ and $18$; for $t\ge5$ they are at most $t+12-\gamma\le17$ and $18$. They too are at most $d+18$.
\end{proof}

\begin{proof}[Proof of Theorem~\ref{thm:main}]
We first dispose of the small-diameter and minimum-radius cases. For the remaining graphs, Lemma~\ref{lem:initial-estimate} either gives the bound directly or leaves a saving in the range where Lemma~\ref{lem:farthest-extraction} applies. We then use the choice with minimum saving described above.

For $d=0$ the assertion is immediate. For $1\le d\le2$, Theorem~\ref{thm:radius} at a centre of eccentricity $\rho\le d$ gives diameter at most $2\rho(\rho+1)\le12<18$. Assume $d\ge3$. If $d=2r$ and $t=0$, the same theorem gives $2r(r+1)=L+d$. Otherwise choose $p,q,Q$ by Lemma~\ref{lem:axis-choice}. Lemma~\ref{lem:initial-estimate} settles the edge-cover alternative. Orient all edges and select the candidate paths and their actual attachments as in Section~\ref{sec:joining-paths}; if the resulting saving satisfies $\delta_0\ge r+t-8$, that lemma also finishes the proof.

Suppose instead that $\delta_0$ lies in \eqref{eq:deficit-range}, and keep the centre $p$ fixed. Consider the initial choices and all choices obtained from them after finitely many applications of Lemmas~\ref{lem:farthest-extraction} and~\ref{lem:rebuilt-core}. Each application starts with a saving in \eqref{eq:deficit-range}; after rebuilding, orient all edges inside and outside $H$, add all permitted improved paths to the candidate families, and select paths at their actual attachments as in Section~\ref{sec:joining-paths}. Include the resulting choices even when their new saving lies outside \eqref{eq:deficit-range}. Each choice includes the farthest vertex, the shortest path $Q$, $H$, all edge directions, the candidate paths, and the selected actual attachments. The initial choices are included, and all savings are nonnegative integers. Therefore a minimum saving $\delta$ is attained. Since $\beta\le\delta\le\delta_0$, the choices attaining this minimum still satisfy \eqref{eq:deficit-range}.

Apply Lemma~\ref{lem:farthest-extraction} to choices attaining this minimum, and rebuild at the resulting farthest vertex $w$. After orienting all edges and adding and selecting the candidate paths as above, the resulting choices are also obtained from the initial ones by finitely many rebuildings. Their saving therefore satisfies $\delta'\ge\delta$. Lemma~\ref{lem:rebuilt-estimate} supplies the required orientation. The centre and radius, and hence $r,t,\varepsilon$, remain fixed. Only the initial choices use Lemma~\ref{lem:axis-choice}; after each rebuilding, Lemma~\ref{lem:rebuilt-distances} applies in the resulting orientation.
\end{proof}

\section*{Acknowledgements}

Y. Chen is supported by the National Key R\&D Program of China (No. 2024YFA1013900) and the National Natural Science Foundation of China (No. 12471327).
J. Ai is supported by the National Natural Science Foundation of China (No. 12522117) and the Fundamental and Interdisciplinary Disciplines Breakthrough Plan of the Ministry of Education of China (JYB2025XDXM207).  H. Lei is supported by the National Natural Science Foundation of China (Nos. 12371351, 12431013), the Natural Science Foundation of Tianjin (No. 24JCYBJC01670), and the Fundamental Research Funds for the Central Universities, Nankai University (No. 63263098). X. Wang is supported by the National Natural Science Foundation of China (No. 12401447).

\section*{AI Declaration}

We used ChatGPT to assist with the exploration of proof strategies and the preparation of this manuscript. The quadratic coefficient $1/2$ was determined independently by Ai, Feng, Lei and Ren~\cite{AiFengLeiRen2026} and by Chen, Lin, Wang and Zheng~\cite{ChenLinWangZheng2026}, using different constructions of a central subgraph. We later combined the two proof methods to determine the linear coefficient $1$ and prepared this paper jointly.
All mathematical arguments and proofs in the final
manuscript were written and verified by the authors.

\end{document}